\documentclass[a4paper,11pt,pdftex]{article}

\usepackage{amsmath}
\usepackage{amssymb}
\usepackage{times}
\usepackage{amsthm}
\usepackage{color}
\usepackage{enumerate}
\usepackage{enumitem}
\setlist[enumerate,1]{label={\normalfont(\roman*)}}
\usepackage{mathrsfs}
\usepackage[pdftex]{graphicx}
\usepackage{float}
\usepackage{wrapfig}
\usepackage{layout}
\usepackage{cite}
\usepackage{braket}
\usepackage
[draft=false, setpagesize=false, pdfstartview=FitH,
colorlinks=true, linkcolor=blue, citecolor=magenta, pagebackref=true, bookmarks=false]
{hyperref}

\usepackage{lineno}
 \newcommand*\patchAmsMathEnvironmentForLineno[1]{
   \expandafter\let\csname old#1\expandafter\endcsname\csname #1\endcsname
   \expandafter\let\csname oldend#1\expandafter\endcsname\csname end#1\endcsname
   \renewenvironment{#1}
     {\linenomath\csname old#1\endcsname}
     {\csname oldend#1\endcsname\endlinenomath}}
 \newcommand*\patchBothAmsMathEnvironmentsForLineno[1]{
   \patchAmsMathEnvironmentForLineno{#1}
   \patchAmsMathEnvironmentForLineno{#1*}}
 \AtBeginDocument{
 \patchBothAmsMathEnvironmentsForLineno{equation}
 \patchBothAmsMathEnvironmentsForLineno{align}
 \patchBothAmsMathEnvironmentsForLineno{flalign}
 \patchBothAmsMathEnvironmentsForLineno{alignat}
 \patchBothAmsMathEnvironmentsForLineno{gather}
 \patchBothAmsMathEnvironmentsForLineno{multline}
 }

 \newcommand{\supp}{\mathrm{supp\,}}

\newcommand{\R}{\mathbb{R}}

\newcommand{\N}{\mathbb{N}}
\newcommand{\e}{\varepsilon}
\newcommand{\rd}{\mathrm{d}}

\def\dist{\mathrm{dist}}

\theoremstyle{definition}
\newtheorem{definition}{Definition}[section]
\theoremstyle{plain}
\newtheorem{theorem}[definition]{Theorem}
\newtheorem{corollary}[definition]{Corollary}
\newtheorem{lemma}[definition]{Lemma}
\newtheorem{proposition}[definition]{Proposition}
\newtheorem{remark}[definition]{Remark}
\newtheorem*{notation}{Notation}

\usepackage[truedimen,margin=15truemm, bottom=20truemm]{geometry}

\usepackage{authblk}
 
\title
	{\bf Normalized clustering peak  solutions for Schr\"odinger equations with general nonlinearities}

  \author[a]{Chengxiang Zhang\thanks{  zcx@bnu.edu.cn}}
  \author[b]{Xu Zhang\thanks{ darkblue1121@163.com, corresponding author  }}

  \affil[a]{\footnotesize Laboratory of Mathematics and Complex Systems (Ministry of Education), School of Mathematical Sciences, 
  
  Beijing Normal University, Beijing 100875, P. R. China}
  \affil[b]{\footnotesize  School of Mathematics and Statistics, Central South University, Changsha 410083, P. R. China}
  
  \date{}
  
\begin{document}
  \maketitle
  \begin{minipage}{16.5cm}
    {\small {\bf Abstract:}
     We are concerned with the normalized $\ell$-peak solutions to  the nonlinear Schr\"{o}dinger equation
    \[
      \begin{cases}
        -\varepsilon^2\Delta v+V(x)v=f(v)+\lambda v,\\
        \int_{\R^N}v^2 =\alpha \e^N. 
        \end{cases}
     \]
    Here $\lambda \in \mathbb{R}$ will arise as a Lagrange multiplier,
      $V$ has a local maximum point, and $f$ is a general $L^2$-subcritical nonlinearity  satisfying a nonlipschitzian property  that $\lim_{s\to0} f(s)/s=-\infty$.
    The peaks of solutions that we construct cluster near a local maximum of $V$ as $\e\to0$.
    Since there is no information about the uniqueness or nondegeneracy for the limiting system,  
    a delicate lower gradient estimate should be established when the local centers of mass of functions are away from the local maximum of $V$.
    We introduce a new method to obtain this estimate, which is significantly different from the ideas in  del Pino and Felmer \cite{DelPino-Felmer2002} (Math. Ann. 2002),
    where a special gradient flow with high regularity  is used, and in Byeon and Tanaka \cite{Byeon-Tanaka,Byeontanaka} (J. Eur. Math. Soc.  2013 \& Mem. Amer.
    Math. Soc. 2014), where an extra  translation flow is introduced.
    We also give  the existence of ground state solutions for the autonomous problem, i.e., the case $V\equiv0$. 
    The ground state energy is not always negative and
    the strict subadditive property of ground state energy  
    here is achieved by strict  concavity.  
    
    
    }
    
    \end{minipage}

\section{Introduction and main Results}
We study the  semiclassical states of the following
logarithmic Schr\"odinger equation
\begin{equation}\label{1.1}
\begin{cases}
-\varepsilon^2\Delta v+V(x)v=f(v)+\lambda v,\\
\int_{\R^N}v^2 =\alpha \e^N, \ \  
\end{cases}
\end{equation}
where $N\geq 1$, $\varepsilon>0$ is a small parameter, $f$ is a general nonlinearity, and $V$ is a function having  a local maximum point.
The problem comes from the study of stationary states for the time-dependent nonlinear Schr\"odinger equation
\begin{equation}\label{NLS}
i\hbar \frac{\partial \psi}{\partial t}=-\frac{\hbar^2}{2m}\Delta \psi +V(x)\psi -g(|\psi|)\psi=0.
\end{equation}
Note that a stationary state possesses the form  $\psi( x, t)= v(x) e^{-\frac{i\lambda t}{\hbar}}$. 
Then $\psi$ is a stationary solution to \eqref{NLS} if and only $(\lambda, v)$ is a solution to \eqref{1.1} with $\e=\frac{\hbar}{\sqrt{2m}}$, $f(u)=g(u)u$.
The  $L^2$ constraint in \eqref{1.1} comes from the mass conservation property of the  stationary state.
Remark   that  solutions under the  $L^2$ constraint are
usually referred to as the normalized solutions.

In the autonomous case $V\equiv V_0$, by a transformation of variable $u(x)=v(\e x)$, and by replacing the unknown number $\lambda$ by $\lambda+V_0$, problem \eqref{1.1} is equivalent to 
\begin{equation}\label{equation1}
  \begin{cases}
    -\Delta u=f(u) +\lambda u,\\
    \int_{\R^N} u^2=\alpha.
  \end{cases}
\end{equation}
 This autonomous problem has been extensively studied
  since   \cite{lions,CL82} in the $L^2$-subcritical case and \cite{Jeanjean97}   
  in the $L^2$-supercritical case. The existence results are built for more general nonlinearities in \cite{Shibata,IT19,Jeanlu} recently.
 On the other hand, 
 the solvability of \eqref{1.1} with various nonconstant potentials and general nonlinearities is rather poorly understood so far.
When $\e=1$,
\cite{IM20,YQZ} give the existence of solutions for $L^2$-subcritical case under the assumption $  \lim_{x\to\infty} V(x)=V_\infty\geq V\not\equiv V_\infty$;  
\cite{Dingzhong} considered  with similar potential assumption and Ambrosetti--Rabinowitz type conditions on nonlinearity in the $L^2$-supercritical case; and
\cite{BMR21} studied the $L^2$-supercritical  problem with a power type nonlinearity $f(u)=u^{p-1}$ and a positive potential    vanishing at infinity.
We also note that \cite{AW19} studied solutions of multibump type with periodic assumptions under a strict nondegeneracy condition.
Considering $\e$ as a small parameter, 
\cite{Alves,zhangzhang}   studied the problem 
with $V=0$ and a potential $K$ on the nonlinearity, i.e., $K(x)f(u)$ or $K(x)u^{p-1}$. If   $K$ has local maximum points
\cite{Alves} showed the existence of local minimizers for $L^2$-subcritical problem ; \cite{zhangzhang} constructed multibump solutions with each bump concentrate to 
a local maximum point of $K$ in the $L^2$-subcritical and $L^2$-supercritical case
by a local deformation argument. We also refer to \cite{TZZZ} in which the authors studied  problems  in bounded set with several component  and problems in the whole space  with a steep well potential.
As the mass tends to some limit, the problems are transformed to  singular perturbed type with two parameter similar to \eqref{1.1}.
However,  there is few result for \eqref{1.1} with a potential $V$ having local maximum points.

For singular perturbation problems without an $L^2$ constraint, that is, for
the following equation
\[
  -\varepsilon^2\Delta v+V(x)v=f(v),\quad v\in H^1(\R^N),
\]
there have been many studies on constructing solutions concentrated near local critical points of the potential following the pioneering work
of Floer and Weinstein \cite{FW}. In  \cite{FW}, they found a  positive solution which concentrates at a nondegenerate critical point of $V$
by the Lyapunov--Schmidt reduction method which requires some  nondegeneracy conditions on the limiting problem.
For problems with  no uniqueness or nondegeneracy condition   assumed  on the  limiting problem,
 the  solutions are usually found as 
critical points of corresponding functionals through the variational approach, in which basic strategy is to obtain a Palais--Smale sequence  
 through a    deformation  
generated by a descending flow, usually the negative gradient flow. 
This method was initiated by Rabinowitz in  \cite{Rabinowitz1992}. 
See also \cite{DelPino-Felmer1996,DelPino-Felmer1997,DelPino-Felmer2002,Byeon2002,BW-1,Byeon-Tanaka,Byeontanaka,WX1,DPR,DPRc,CingolaniTanaka} for further studies.

The motivation 
for this paper is that  
the known studies on \eqref{1.1} or \eqref{equation1} are mainly based on an assumption that $f(s)/s\to 0$ as $s\to 0$. This excludes some nonlipschitzian nonlinearity
such as $s\log s+ u^{p-1}$,  or $-u^{q-1}+u^{p-1}$,  where $q\in(1,2)$ and $p\in(2,2+\frac4N)$. We first study the autonomous problem 
and consider a general class of nonlinearities such that $f(s)/s\to-\infty$ as $s\to 0$.
More precisely, we impose the following assumptions on  $f$:
\begin{enumerate}
  \item[(F1)] $f \in C(\mathbb{R},\R )$ and $f(0)=0$.
  \item[(F2)]  $\lim _{s \rightarrow 0^+} f(s) / s =-\infty$.
\item[(F3)] $\limsup_{s \rightarrow +\infty} f(s) / s^{1+4/N}=c_0  $.
   \item[(F4)] 
     $s^{-1}f(s)$ is strictly increasing for $s>0$.
\end{enumerate}
Note that $(F4)$ implies   $c_0\geq 0$.
Ground states  are usually found by the following minimization problem
\begin{equation}\label{Ea}
 E_\alpha= \inf\Big\{J(u)=\frac12\int_{\R^N}|\nabla u|^2-\int_{\R^N}F(u)\ \Big|\ {u\in \mathcal M_\alpha} \Big\},
\end{equation}
where $\mathcal M_\alpha=\set{u\in H^1(\R^N) | \int_{\R^N} u^2=\alpha }$, $F(s)=\int_{0}^{|s|} f(\tau) d \tau$.
It is well known that the following Gagliardo--Nirenberg inequality plays an important role to determine whether the given infimum is well-defined,  
\begin{equation}\label{GN}
  |u|_{2+4/N}^{2+4/N} \leq S(N) |\nabla u |_2^{2} |u |_2^{4 / N}, \quad u\in H^1(\R^N),
\end{equation}
where $S(N)>0$ is the optimal constant for the Gagliardo--Nirenberg inequality.
We set 
 \[\alpha_N:=\begin{cases}
  (2c_0 S(N))^{-\frac N2},\quad &\mbox{if}\quad c_0>0,\\
  +\infty,\quad &\mbox{if}\quad c_0= 0.
 \end{cases}\]

\begin{theorem}\label{them1.1}
  Assume (F1)--(F4). For each  $\alpha\in (0, \alpha_N)$, \eqref{equation1} has a solution $(\lambda, u)$, such that $u$ is a nonnegative nontrivial function, and  is a global minimizer for $E_\alpha$. Moreover, 
  \begin{enumerate}
    \item $E_\alpha$ is continuous and strictly concave.
    \item  $\lim_{\alpha\to 0} E_\alpha =0$ and $E_\alpha>0$ for small $\alpha$.
    \item Assume further  $c_0=0$. Then      $E_\alpha$ has a unique zero in $(0,+\infty)$  and 
    $\lim_{\alpha\to+\infty} E_\alpha=-\infty$ if $f$ admits a zero in $(0,+\infty)$; and 
    $E_\alpha$ is strictly increasing in $(0,+\infty)$ if  $f$ is negative in $(0,+\infty)$.
  \end{enumerate}
\end{theorem}
By the classical result of \cite{CL82} for $L^2$ subcritical problems,
the attainability for the minimization problem \eqref{Ea}, in some sense, is equivalent to the strict subadditive inequality 
\[
E_{\alpha+\beta}<E_{\alpha}+E_{\beta}. 
\]
In \cite{Shibata}, for a class of general Berestycki-Lions type nonlinearities (\cite{bl83-1,bl83-2}) such that  
 $f(s)/s\to 0$ as $s\to 0$,  the energy proved to be nonpositive and nonincreasing. Moreover, it seems that the strict subadditivity holds  only when the energy  $E_{\alpha+\beta}$  is negative.  
In our setting, there is a difference that $E_\alpha$ is positive and strict increasing for small $\alpha$.
Our strategy to obtain the strict subadditivity is to use the strong concavity of $E_\alpha$, that is the merit of   (F4).

Next we study \eqref{1.1}. We will construct normalized solutions with $\ell$-peaks if the potential $V$ has a local maximum point.
In light of \cite{zhangzhang} and \cite{TZZZ},
the following limiting system for \eqref{1.1} is important
\begin{equation}\label{eqsystem}
  \begin{cases}
  -\Delta u_j=f(u_j)+\lambda u_j \  \ \text {in}\ \ \mathbb R^N,\\
  u_j(x)>0,\  \lim_{|x|\to\infty}u_j(x)=0,\quad i=1,2, \cdots, \ell\\
  \sum_{i=1}^{\ell}|u_j|_2^2=\alpha.
  \end{cases}
  \end{equation}
It is clear that this system \eqref{eqsystem} has a solution $(\lambda, u_1,\cdots, u_\ell)$ by Theorem \ref{them1.1} by setting $u_i\equiv u_0$ for a solution $(\lambda, u_0)$ to \eqref{equation1} with $\int_{\R^N}u_0^2=\ell^{-1}\alpha$. However, there is no uniqueness or nondegeneracy result  for this solution. In fact, we are even not sure that whether a solution $(\lambda, u_1,\cdots, u_\ell)$ to \eqref{eqsystem} would satisfy $\int_{\R^N}u_i^2=\ell^{-1}\alpha$ for each $i=1,\cdots,\ell$. To state our result, we give the assumptions on $V$ precisely,
\begin{enumerate}
	\item[(V1)] $V(x)\in C(\mathbb R^N)$   and $\liminf_{|x|\to\infty}V(x)|x|^{-2}>-\infty$;
\item[(V2)] There is a bounded domain
  $\Omega\subset {\mathbb{R}}^{N}$  such that $V\in C^1(\overline\Omega)$ and $$V_0:=\max\limits_{x\in \overline\Omega}V(x)>\max\limits_{x\in \partial\Omega}V(x);$$
  \item[(V3)] Let $\mathcal{V}=\left\{x \in \Omega \mid V(x)=V_{0}\right\} .$ Then for any open neighborhood $\widetilde{O}$ of $\mathcal{V},$ there exists an open set $O \subset \widetilde{O}$ such that
  $$
  \mathcal{V} \subset O \subset \overline{O} \subset \widetilde{O} \cap \Omega \quad \text{and}\quad
  \inf _{x \in \partial O}|\nabla V(x)|>0.
  $$
\end{enumerate}

To construct  solutions with $\ell$-peaks, we need another technical condition on the  nonlinearity.
 \begin{enumerate}
    \item[(F5)] $f\in C^1(0,+\infty)$, and for some $\sigma>0$   there hold
      \[ \limsup_{s \to 0^{+}} \Big|f'(s)-\sigma\log s\Big|<+\infty.\]
 \end{enumerate}
We show the following result.
\begin{theorem}\label{th1.1}
Suppose that (F1)--(F5) and  (V1)--(V3) hold. 
For any $\alpha\in (0, \alpha_N)$,  $\ell\in \N\setminus\{0\}$, there exists $\varepsilon_\ell>0$ such that for each $\varepsilon\in(0,\varepsilon_\ell)$, equation \eqref{1.1} admits a solution $(\lambda_\e, v_\e)$  satisfying 
\begin{enumerate}
  \item $v_\varepsilon>0$ has exact $\ell$ peaks  
  $x_{\varepsilon}^{1}, \cdots, x_{\varepsilon}^{\ell} \in \R^{N}$ satisfying $\displaystyle
\lim _{\varepsilon \to 0} \operatorname{dist}\left(x_{\varepsilon}^{j}, \mathcal{V}\right)=0 \  \text { for all } j \in\left\{1, \cdots, \ell\right\},
$
\item setting $u_{\varepsilon}(x)=v_{\varepsilon}( \varepsilon x),$ there exist a subsequence $\varepsilon_{j} \to 0$ such that
$$ \lambda_\e \to \lambda +V_0,\quad \mbox{and}\quad 
\left\|u_{\varepsilon_{j}}-\sum_{k=1}^{\ell} u_j\left(\cdot-x_{\varepsilon_{j}}^{k} / \varepsilon_{j}\right)\right\|_{H^{1}} \to 0 \quad \text { as } j \to \infty, 
$$
where $(\lambda, u_1,\cdots,u_\ell)\in \R\times H^1(\R^N)^\ell$ is a   solution   to the system 
\eqref{eqsystem}.

\item 
there exist  $C,c>0$ such that
$$v_\varepsilon(x)\leq C\sum_{j=1}^\ell e^{-c\varepsilon^{-2}|x-x^j_\e|^2}\quad \text{for}\ \ x\in\mathbb R^N.$$
\end{enumerate}

\end{theorem}

To find critical points in a neighborhood of the approximate solutions,
following the idea of  \cite{Coti1,Coti,sere}, a crucial step to make  deformation  is to obtain a uniform gradient estimate for the functional in an annular domain, 
i.e., a uniform lower bound for the norm of gradient of the functional in an annular domain. 
The uniform gradient estimate can be obtained when we search for critical points near local minimum points of $V$.
This is because, by the characteristic of local minimum and   monotonicity property of least energy for the limiting problem,
the functions near the approximate solutions with energy no greater than the least energy will concentrate 
to the local minimum of $V$.
The situation becomes   more  complicated for general  saddle points or maximum points.
Actually, the repelling property of such critical points makes it impossible to obtain the uniform gradient estimate
since the barycenters (or local centers of mass) of functions near the approximate solutions will tend to deviate from the critical points to decrease its energy.
Here we refer to \cite[CHAPTER 6]{Byeon-Tanaka} for a counterexample in this case.
Therefore, another much more delicate lower gradient estimate should be obtained for functions whose barycenters are away from 
the critical point,
 in order that barycenters of functions along the negative gradient flow would not move too far away
 before the energy is deformed to a given lower level.

We explain two methods from \cite{DelPino-Felmer2002} and \cite{Byeontanaka,Byeon-Tanaka}  to deal with this difficulty in nonlinear Schr\"odinger equations  without $L^2$ constraint.
In del Pino and Felmer \cite{DelPino-Felmer2002},  
that lower gradient estimate is obtained for the energy functional only at functions having a uniform $H^2$ bound. 
Thus, the authors defined a special negative gradient flow on the Nehari manifold, and they are able to 
  show the uniform $H^2$ bounds for functions along the flow if the flow starts from a suitable
  test path with a well-chosen set of initial conditions.
 Another method is developed by Byeon and Tanaka in \cite{Byeontanaka,Byeon-Tanaka}. They introduced another decreasing flow, i.e., the translation flow generated by 
  $-\nabla V$.
  They  are able to bypass the obstacle in obtaining the lower gradient estimate through several steps of iterations among 
  the negative pseudogradient flow of the energy functional, the  tail-minimizing
  operator that keeps tails small, and the translation flow.

   There are essential difficulties in applying those two methods in our setting. First, it is important    
   to obtain the global  $H^2$ regularity uniformly for the special flow in \cite{DelPino-Felmer2002}. 
   This relies on some  stronger conditions   on the nonlinearity $f$.  
     However, the nonlinearity in this paper having non-lipschitzian properties is not good enough for us to obtain the global  $H^2$ regularity uniformly 
     along the pseudogradient flow.  See Remark \ref{rek4.3} (i) for more discussions. 
 On the other hand, although the arguments in \cite{Byeontanaka,Byeon-Tanaka}  works well for   nonlinear Schr\"odinger equations without an $L^2$-constraint
 in very weak conditions, it heavily depends on the use of an tail-minimizing operator, 
 which is defined by solving a local minimization problem in an exterior domain with some prescribed boundary condition.
However, in the situation with an $L^2$ constraint, it is difficult to perform the minimization
argument locally on the $L^2$ spheres.

  In this paper, we will develope another approach to deal with this problem. 
  In fact, at every function whose local centers of mass are away from the 
  maximum points of $V$,
  we are able to obtain the desired lower  gradient estimate, 
  without assuming uniform $H^2 $ bounds on these functions.
We explain  our strategy  
    as follows.
We first introduce a new penalization functional  so that 
a priori decay estimate in some  exterior region away from the local centers of mass of the functions can be obtained.   
This estimate implies that  the exterior norms of the functions can be controlled by  
the gradient of the  energy functional.
Thus,  we can get 
rather fine decay for the functions which do not meet the desired lower  gradient estimate.
We are able to find good replacements of these functions by introducing an elliptic equation which is defined  by a minimization problem in a hyperplane of the Sobolev space.
 Then a contradiction could be obtained by a check of balance of the elliptic equation.
 We remark that our idea applies likewise to nonlinear Schr\"odinger equations without $L^2$-constraints under very weak conditions on the nonlinearity. In fact,
 it works well to the situation of \cite{Byeontanaka,Byeon-Tanaka}. We will explain it later in Remark \ref{rek4.3} (ii).

 At last, we mention that our assumptions (F1)--(F5) cover the nonlinearity of   logarithmic type. By  shift invariant property of the logarithmic Schr\"odinger equation (see \cite{Tanaka-Zhang,ITWZ}), we can 
 give a multiplicity result in the setting without an $L^2$-constraint.
  By Theorem \ref{th1.1}, it is easy to verify that $w_\e = e^{\lambda_\e/2}v_\e$ is a solution to 
    the following logarithmic Schr\"odinger equation (without an $L^2$ constraint condition):
  \begin{equation}\label{eq7777}
    -\varepsilon^2 \Delta w+ V(x)w=w\log w^2,\quad w\in H^1(\R^N).
  \end{equation}
  \begin{corollary}\label{cor1}
    Assume (V1)--(V3). Then for any $\ell\in \mathbb N\setminus\{0\}$, there is $\e_\ell>0$ such that for each $\varepsilon\in(0,\varepsilon_\ell)$, equation \eqref{eq7777} admits a solution $  v_\e$ positive solution with $\ell$ peaks, which concentrate to $\mathcal V$ as $\e\to 0$. 
  \end{corollary}
We also comment  that our assumptions on $V$ includes a class of strong repulsive potentials, for example, $V(x)=-|x|^2$ (see \cite{zhang1,CS}). Corollary \ref{cor1}  in fact gives the existence result of multiple nonradial solutions  for such repulsive potentials when $\e$ is small.

\begin{notation}
   Throughout this paper,
$2^*=+\infty$ for $N=1,2$ and $2^*=\frac{2N}{N-2}$ for $N\geq 3$; $ L^p(\mathbb R^N) \ (1\leq p<+\infty)$ is the usual Lebesgue space with the norm $ |u| _p^p=\int_{\mathbb R^N}|u|^p;$
$ H^1(\mathbb R^N)$ denotes the Sobolev space with the norm $\|u\|^2=\int_{\mathbb R^N}(|\nabla u|^2+|u|^2);$
$o_n(1)$ (resp. $o_\varepsilon(1)$)
will denote a generic infinitesimal as $n\rightarrow \infty$ (resp. $\varepsilon\rightarrow 0^+$);  $B (x,\rho)$ denotes an open ball centered at $x\in\mathbb R^N$ with radius
$\rho>0$.
$a^\pm=\max\{0,\pm a\}$ for $a\in\mathbb R$.
 Unless stated otherwise,   $C, C'$ and $c$ are general constants.
 
\end{notation}

\section{The least energy for the  autonomous problem}
In this section, we solve the following minimization problem 
\begin{equation}
 E_\alpha= \inf\Big\{J(u)=\frac12\int_{\R^N}|\nabla u|^2-\int_{\R^N}F(u)\ \Big|\ {u\in \mathcal M_\alpha} \Big\},
\end{equation}
where $\mathcal M_\alpha=\set{u\in H^1(\R^N) | \int_{\R^N} u^2=\alpha }$, $\alpha\in(0,\alpha_N)$, $F(s)=\int_{0}^{|s|} f(\tau) d \tau$, and $f$    satisfies
(F1)--(F4).
We first note that under the assumption (F1), either of  the following conditions is equivalent to (F4). 
  \begin{enumerate}
    \item[(F4')] The function $t\mapsto F(\sqrt{t})$ is strictly convex for $t>0$.
    \item[(F4'')] 
     $F(\sqrt{1-s}u)+F(\sqrt{1+s}u)>2F(u)$ for $s\in(0,1)$, $u\neq 0$.
  \end{enumerate}
  In fact, 
  if (F4) holds, we have
  \begin{equation}
   \frac{\rd}{\rd s}\left(F(\sqrt{1-s}u)+F(\sqrt{1+s}u)\right)=\frac{u^2}2\left(\frac{f(\sqrt{1+s} u)}{\sqrt{1+s} u}
   -\frac{f(\sqrt{1-s}u)}{\sqrt{1-s}u}\right)>0.
  \end{equation}
  Then (F4'') follows from (F4).
On the other hand, (F4'')
implies that the  function $t\mapsto F(\sqrt{t})$ is strictly midpoint convex. Thus, it is strictly convex by continuity.
Hence, (F4'') implies (F4'). At last, (F4')
implies that $\frac{\rd}{\rd t} F(\sqrt{t})= \frac{f(\sqrt{t})}{2\sqrt{t}}$ is strictly increasing, which is exactly (F4).

By (F2)  and (F4), $f$ admits at most one zero  in $(0,+\infty)$. Hence, we   set   $t_0=+\infty$ 
if  $f$ is negative in $(0,+\infty)$, and $t_0$ to be
the unique zero of $f$ if $f$ changes its sign in $(0,+\infty)$. 
 We set
 \[ 
  f_1(t)=\begin{cases}
    f^-(t), &t\geq 0\\
    - f^-( -t), & t <0,
  \end{cases}
  \quad 
  f_2(t)=\begin{cases}
    f^+(t), &t\geq 0,\\
    - f^+( -t), & t<0,
  \end{cases}
  \]
  \[ 
  F_1(t)=\int_0^{t} f_1(s) \rd s,
  \quad 
  F_2(t)=\int_0^{t} f_2(s) \rd s.
  \]
  Then 
  \begin{equation}\label{HG}
    F_1(t)= \begin{cases}
      -F(t),& |t|\in[0, t_0),\\
      -F(t_0), &|t|\in[t_0,+\infty),
    \end{cases}
    \quad 
    F_2(t)= \begin{cases}
      0,& |t|\in[0, t_0),\\
      F(t)-F(t_0), &|t|\in[t_0,+\infty).
    \end{cases}
    \end{equation}
    Remark that  $F_1(t)=-F(t)$ and $F_2(t)=0$ in the case $t_0=+\infty$.

    \begin{lemma}\label{F12} 
      Assume (F1)--(F4),
      The following statements hold.
      \begin{enumerate}
        \item For $t>0$, $F_1(\sqrt t)$ is nondecreasing and concave, and $F_2(\sqrt t)$ is nondecreasing and convex.
        \item There is $C>0$ such that for each $t>0$
        \begin{equation}\label{5}
          f(t)\leq f_2(t) \leq C t^{1+\frac 4 N} \quad \mbox{and}\quad F(t)\leq F_2(t) \leq C t^{2+\frac 4 N}.
        \end{equation}
        Moreover, for any $\tau>0$ there is $C_\tau>0$ such that
        \begin{equation} \label{6}
         f(t) \leq   (c_0+\tau) t^{1+\frac 4 N} +C_\tau t  \quad \mbox{and}\quad
          F(t)\leq  (c_0+\tau) t^{2+\frac 4 N} +C_\tau t^2.
        \end{equation}
        \item $t\mapsto F(t)/t^2$ is strictly increasing for $t>0$ and  $f(s)s>2F(s)$   a.e. $s\in\R\setminus\{0\}$.
        Similarly,
        $t\mapsto F_1(t)/t^2$ is   nonincreasing for $t>0$ and  $f_1(s)s\leq 2F_1(s)$,    $s\in\R$.
         
      \end{enumerate}
      
    \end{lemma}
    \begin{proof}
      (i) follows from the definition of $F_1$ and $F_2$, and (F4'). (ii) follows from (F2)(F3). 
      By (F4'), we have
       \begin{equation}\label{AR1}
         t^2 F(u)= t^2 F(\sqrt{t^{-2}(tu)^2+(1-t^{-2})0}) < F(tu)+( t^{2}-1)F(0)=F(tu), \quad \mbox{for}\quad t>1, u\neq 0,
       \end{equation}
       implying that $t\mapsto F(t)/t^2$ is strictly increasing for $t>0$.
      Differentiating  $F(t)/t^2$, we know  $f(t)t>2F(t)$   a.e. $t>0$. This inequality holds almost everywhere in  $\R$ by symmetry.
    \end{proof}

By \eqref{6} and the Gagliardo--Nirenberg inequality \eqref{GN},
it is clear that  $E_\alpha$ is well-defined for each $\alpha\in (0, \alpha_N)$, where $\alpha_N:=(2c_0 S(N))^{-\frac N2}$ if $c_0>0$, $\alpha_N=+\infty$ if $c_0=0$.
    \begin{lemma}
      If $E_\alpha$ is attained by some  $u$, then $f(u)u\in L^1(\R^N)$ and $u$
      satisfies
      \[
      -\Delta u=f(u)+\lambda u,  
      \]
      where
      \[
        \lambda =\alpha^{-1}\left(\int_{\R^N}|\nabla u|^2-\int_{\R^N} f(u)u\right).
        \]
    \end{lemma}
    \begin{proof}
      By Lemma \ref{F12} (iii), 
      $\int_{\R^N} f(u)u\geq 2\int_{\R^N} F(u)>-\infty$. By this and \eqref{5}, $f(u)u\in L^1(\R^N)$.
Note that 
\[
E_\alpha \leq   E_{n,\alpha} :=\Set{J(u) | u\in H_0^1(B_n), \int_{B_n} u^2=\alpha }.
\]   
Taking $\varphi\in C_0^\infty (B_1)$ with $0\leq \varphi\leq 1$ in $B_1$, $\varphi=1$ in $B_{1/2}$,
we set 
\[
u_n= \alpha^{1/2} |\varphi(n^{-1}\cdot) u|_2^{-1} \varphi(n^{-1}\cdot) u.
\]
Then it is easy to check that 
\[
  u_n\to u\quad \mbox{in}\ H^1(\R^N),\quad \frac12\int_{\R^N}|\nabla u_n|^2-\int_{\R^N}F_2(u_n)\to \frac12\int_{\R^N}|\nabla u|^2-\int_{\R^N}F_2(u).
\]
On the other hand, 
since $|u_n|\leq 2 |u|$, we have  $F_1(u_n)\leq F_1(2 u)$.
Similarly to \eqref{AR1}, $F_1(2u)\leq 4F_1(u) \in L^1(\R^N)$.
Hence, by the Lebesgue convergence theorem,
\[
  E_{n,\alpha}\leq  J(u_n)\to E_\alpha .  
\]

By the Ekeland variational principle, there is $\lambda_n\in\R$ such that 
\[
\|J'(u_n)-\lambda_n u_n\|_{H^{-1}(B_n)}\to 0.  
\]
Since 
$f_1(u_n)u_n\leq 2F_1(u_n)$, we can conclude that 
\[\lambda_n|u_n|_2^2=J'(u_n)u_n+o_n(1)\to \frac12\int_{\R^N}|\nabla u|^2-\int_{\R^N} f(u)u.\]
 Hence, 
 \[
  \lambda_n \to \lambda =\alpha^{-1}\left(\frac12\int_{\R^N}|\nabla u|^2-\int_{\R^N} f(u)u\right).
 \]
 On the other hand, for any $\varphi\in C_0^\infty(\R^N)$, we have  $\supp\varphi\subset B_n$ when   $n$ is sufficiently large,
and hence
 \[
  J'(u_n)\varphi-\lambda_n \int_{B_n}u_n \varphi\to 0. 
 \]
Thus, 
$u$ solves $-\Delta u =f(u)+\lambda u$.
    \end{proof}

Note also that 
\[
J(u)\geq \frac12|\nabla u|_2^2 -\int_{\R^N} F_2(u).
\]
  Therefore,
\[E_\alpha\geq \widehat E_\alpha:=\inf\Big\{ \frac12\int_{\R^N}|\nabla u|^2-\int_{\R^N}F_2(u)\ \Big|\ u\in H^1(\R^N), \quad \int_{\R^N} u^2=\alpha  \Big\}.
\]
 
\begin{lemma} \label{Lemma:2.1}
The following statements hold.
  \begin{enumerate}
    \item $E_\alpha$ is nonnegative for small $\alpha$. 
    \item $\alpha\mapsto E_\alpha$ is midpoint concave in $(0,\alpha_N)$, i.e., for any  $\alpha\in(0,\alpha_N)$ and $\theta\in(0,1)$ with $\alpha+\theta\alpha\in (0,\alpha_N)$,
    \begin{equation}\label{midconcave}\frac12(E_{\alpha-\theta\alpha}+E_{\alpha+\theta\alpha})\leq E_{\alpha}.
    \end{equation}
    If $E_\alpha$ is attained for some $\alpha_0>0$, then the inequality \eqref{midconcave} is strict for $\alpha_0$ and every $\theta\in(0,1)$ with $\alpha_0+\theta\alpha_0\in (0,\alpha_N)$.
    \item $\alpha\mapsto E_\alpha$ is   continuous and concave   in $(0,\alpha_N)$, and $\lim_{\alpha\to 0} E_\alpha =0$.
     
  \end{enumerate}

\end{lemma}
\begin{proof}
(i)
By \eqref{5} and the Gagliardo--Nirenberg inequality, we have 
\[J(u)\geq \frac12 |\nabla u|_2^2-C \int_{\R^N}|u|^{2+\frac4N}\geq (\frac12 -C(N)\alpha^{\frac4N})|\nabla u|_2^2,\quad\mbox{where}\ |u|_2^2=\alpha.
\]
Then we can conclude that $E_\alpha\geq 0$ if $\alpha$ is sufficiently small.

(ii)
Let $\alpha >0$ and $\theta\in(0,1)$.
Assume $\{u_n\} \subset H^1(\R^N)$ is such that 
$$J(u_n)\leq E_{\alpha}+\frac1n,\quad  |u_n|_2^2=\alpha.
$$
Then by (F4''),
\[
\begin{aligned}
E_{\alpha-\theta\alpha}+E_{\alpha+\theta\alpha}\leq& J\left(\sqrt{1-\theta} u_n\right)+J\left(\sqrt{1+\theta}u_n\right)\\
  =& |\nabla u_n|_2^2-\int_{\R^N}\left(F\left(\sqrt{1-\theta} u_n\right)+F\left(\sqrt{1+\theta}u_n\right)\right)\\
  <&2J(u_n)\leq 2E_{\alpha}+\frac2n.
\end{aligned}  
  \]

Hence, letting $n\to\infty$, we have the midpoint concavity.
Moreover, if $E_\alpha$ is attained, then we   just choose $u\in H^1(\R^N)$ such that $J(u)= E_{\alpha}$ and  $|u|_2^2=\alpha$.
    Hence,  the inequality holds strictly.

(iii)    To see  $E_\alpha$ is continuous and concave,  it suffices to show that $E_\alpha$ is bounded on some interval (see \cite{Dono}).
 Since   
 $f_2(s)$ either identically zero   or   satisfies the assumptions in \cite[Lemma 2.3]{Shibata}, we conclude that the function $\alpha\mapsto\widehat E_\alpha$ is continuous in $(0,+\infty)$.
On the other hand,
let 
$u_\alpha=\sqrt{\alpha}u_1$, where 
$u_1\in C_0^\infty(\R^N)$ is chosen such that $|u_1|_2^2=1$.
We have 
\[ E_\alpha\leq 
J(u_\alpha)=\alpha \int_{\R^N}|\nabla u_1|^2- \int_{\R^N} F(\sqrt{\alpha}u_1)\leq \alpha \int_{\R^N}|\nabla u_1|^2+ \int_{\R^N} F_1(\sqrt{\alpha}u_1).  
\] 
Hence, $E_\alpha$ is bounded in any finite subinterval of $(0,+\infty)$.
Then $E_\alpha$ must be continuous and   concave in $(0,+\infty)$.
Note   that  for $\alpha\in(0,1)$, $0\leq F_1(\sqrt{\alpha}u_1)\leq F_1(u_1)$. By Lebesgue convergence theorem 
$\lim_{\alpha\to 0^+}\int_{\R^N} F_1 (\sqrt{\alpha} u_1)=0$.
Therefore,
$\lim_{\alpha\to0^+} E_\alpha=0$.
 \end{proof}
By Lemma \ref{Lemma:2.1} (iii), we have:
\begin{lemma}
  \label{Lemma:2.2}
  Let $\alpha,\beta>0$ and $t>1$. Then  $E_{t\alpha}\leq t E_\alpha$ for $t\alpha\in (0,\alpha_N)$,   and 
  $E_{\alpha+\beta}\leq E_\alpha +E_\beta$ for $\alpha+\beta\in (0,\alpha_N)$. Both inequalities hold strictly if $E_\alpha$ is attained.
\end{lemma}
\begin{proof} By concavity, 
  $E_{\alpha+(1-t^{-1})\beta}\geq t^{-1}E_{t\alpha} +(1-t^{-1})E_\beta$ 
  for $t\geq 1$, $\alpha,\beta>0$.
  Letting $\beta\to 0$, we have $E_{t\alpha}\leq t E_\alpha$.
  Then setting $t=1+\frac{\beta}\alpha$, we have 
  \begin{equation}\label{Eab}
    \alpha E_{\alpha+\beta}\leq (\alpha+\beta) E_\alpha.
  \end{equation}
  Interchanging $\alpha$ and $\beta$, we have $\beta E_{\alpha+\beta}\leq (\alpha+\beta) E_\beta$. Hence,  $E_{\alpha+\beta}\leq E_\alpha +E_\beta$.

 Now we assume further that $E_\alpha$ is attained.
  We can take  $\delta\in (0,\alpha)$ so that by Lemma \ref{Lemma:2.1} (ii)
   $$E_{2\alpha}\leq E_{2\alpha-\delta}+E_\delta<2E_\alpha.$$ 
   When $t= 3, 4, \cdots$, 
   we have 
   $$E_{t\alpha}\leq E_{(t-2)\alpha}+E_{2\alpha}<tE_\alpha.$$ 
  When $t\in(1,2)$, we have  
  \[ 
     E_{(2-t)\alpha}+E_{t\alpha}<2 E_\alpha\quad \mbox{and}\quad
    (2-t) E_{\alpha}=(2-t) E_{(2-t)^{-1}(2-t)\alpha}\leq  E_{(2-t)\alpha}.
  \]
Hence, $E_{t\alpha}<tE_\alpha$ for $t\in(1,2)$.
On the other hand, when $t\in(k+1,k+2)$, $k=1,2,\cdots$, we have 
\[
  E_{t\alpha}\leq E_{(t-k)\alpha}+E_{k\alpha}<t E_\alpha.
\]
 Hence, $E_{t\alpha}<tE_\alpha$ for any $t>0$.
 As a result, \eqref{Eab} holds strictly. Then $E_{\alpha+\beta}<E_\alpha+E_\beta$.
\end{proof}

\begin{lemma}\label{Brezislieb}Assume $\int_{\R^N} F_1(u_n) +F_2(u_2)$ is bounded. The following statements hold 
  \begin{enumerate} 
    \item 
    If $|u_n|_{2+\frac4N}\to 0$, then
    $|u_n|_2\to0$.
    \item 
    If $|u_n|_{2+\frac4N}$ is bounded and  $u_n\to u$ a.e., then $F_1(u)\in L^1(\R^N)$ and 
   \[
     \int_{\R^N}F_1(u_n)-\int_{\R^N}F_1(u_n-u)\to \int_{\R^N}F_1(u),\quad \int_{\R^N}F_2(u_n)-\int_{\R^N}F_2(u_n-u)\to \int_{\R^N}F_2(u).
   \] 
  \end{enumerate}
  
\end{lemma}
\begin{proof}
  (i) By (F2),
  for any $\tau>0$, there is $\delta>0$ such that 
  $f_1(t)>\tau^{-1} t$ and $F_1(t)>\frac12\tau^{-1} t^2$ for $t\in (0,\delta)$.
Then 
\[ \int_{\R^N} u_n^2=\int_{|u_n|<\delta} u_n^2 +\int_{|u_n|\geq \delta}u_n^2\leq 2\tau \int_{\R^N}F_1(u_n)+\delta^{-\frac4N}\int_{\R^N}|u_n|^{2+\frac4N}.
  \]

  Hence, 
  \[\limsup_{n\to\infty}\int_{\R^N} u_n^2\leq 2\tau \limsup_{n\to\infty} \int_{\R^N}F_1(u_n).\]\
  This completes the proof.

  (ii) We only show the   result for $F_1$, because the result for 
  $F_2(\cdot)$ follows directly from the Brezis--Lieb lemma.
Since $F_1(\sqrt \cdot)$ is   concave in $(0,+\infty)$ and $F_1(0)=0$, similar to the proof of Lemma \ref{Lemma:2.2},  we have
  for $r>1$, $t>0$ and $s>0$,
  \[
    F_1(\sqrt{rt})\leq rF_1(\sqrt t) \quad \mbox{and}\quad F_1(\sqrt{t+s})\leq F_1(\sqrt{t})+F_1(\sqrt{s}).
  \]
  Now by the inequalities above and the nondecreasing property, for each $s, t>0$ and $\tau\in(0,1)$   we have
  \[
   F_1( {t+s})=F_1(\sqrt{(t+s)^2}) \leq F_1(\sqrt{(1+\tau) t^2+ (1+\tau^{-1}) s^2)})\leq (1+\tau) F_1(t)+   (1+\tau^{-1}) F_1(  s).
  \]
  Hence, when $st\geq 0$
\begin{equation}\label{st+}
 0\leq F_1(t+s)-F_1(t)  \leq \tau F_1(t)+ (1+\tau^{-1})F_1(s). 
\end{equation}
If $st<0$, then 
\[
  F_1( {t+s})=F_1(\sqrt{(|t|-|s|)^2}) \leq F_1( \sqrt{t^2+  s^2})\leq   F_1(t)+     F_1(  s).
 \]
When $st<0$ with $|s|\geq |t|$, we have 
\begin{equation}\label{st}
  -F_1(s)\leq -F_1(t)\leq F_1( {t+s})-F_1(t)\leq F_1(s).
\end{equation}
On the other hand,
when $st<0$ with $|t|>|s|$, we have 
\[
  F_1(t)=F_1(|t|-|s|+|s|)\leq (1+\tau)F_1(|t|-|s|)+ (1+\tau^{-1}) F_1(s)\leq F_1(|t|-|s|)+\tau F_1(t)+(1+\tau^{-1}) F_1(s).
\]
Then for $st<0$ with $|t|>|s|$,
\begin{equation}\label{ts}
  -\tau F_1(t)-(1+\tau^{-1}) F_1(s) \leq F_1(|t|-|s|)-F_1(t) = F_1(t+s)-F_1(t)\leq F_1(s).
\end{equation}
By \eqref{st+}, \eqref{st} and \eqref{ts}, we have for each $s,t\in\R$,
\[
  |F_1(t+s)-F_1(t)  |\leq \tau F_1(t) +(1+\tau^{-1})F_1(s).
\]
Then $F_1$ satisfies the assumption of the general Brezis--Lieb lemma (\cite[Theorem 2]{brezislieb}).
\end{proof}

Now we are ready to show Theorem \ref{them1.1}.
\begin{proof}[Proof of Theorem \ref{them1.1}]
  Let $u_n\in H^1(\R^N)$ be such that
  $|u_n|_2^2=\alpha$ and $J(u_n)\to E_\alpha$.
  Then by Gagliardo--Nirenberg inequality \eqref{GN} and (F3), $\{u_n\}$ are bounded in $H^1(\mathbb R^N)$.
We claim that 
\[
\limsup_{n\to\infty}\sup_{B_1(y)}  |u_n|^2>0.
\]
Otherwise, by Lion's lemma,
$|u_n|_{2+\frac4N}\to 0$. However, $\int_{\R^N}F_1(u_n)\leq J(u_n)+\int_{\R^N}F_2(u_n)$ is bounded. Then we have $|u_n|_2\to 0$, a contradiction.
Now assume, there is $y_n\in\R^N$ such that, up to a subsequence, 
$u_n(\cdot-y_n)\rightharpoonup u$ for some $u\in H^1(\R^N)\setminus\{0\}$.
Setting $v_n=u_n(\cdot-y_n)-u$, $\beta=|u|_2^2\leq \alpha$, we have 
\[
|v_n|_2^2 \to \alpha-\beta,\quad   I(v_n)\to E_\alpha-I(u)\leq E_\alpha-E_\beta.
\]
When  $\beta<\alpha$ and $I(u)=E_\beta$, 
then $E_\beta$ is attained and 
\[
  E_{\alpha-\beta}=\lim_{n\to\infty} E_{|v_n|_2^2}\leq \lim_{n\to\infty} I(v_n)= E_\alpha-E_\beta<E_{\alpha-\beta}. 
\]
When 
$\beta<\alpha$ and $I(u)>E_\beta$,
then 
\[
  E_{\alpha-\beta}=\lim_{n\to\infty} E_{|v_n|_2^2}\leq \lim_{n\to\infty} I(v_n)=E_\alpha-I(u)< E_\alpha-E_\beta\leq E_{\alpha-\beta}. 
\]
In either case, we have a contradiction.
Hence, $\beta=\alpha$, and we have 
$u_n(\cdot-y_n)\to u$ in $L^2(\R^N)$. So
$u_n(\cdot-y_n)\to u$ in $L^{2+4/N}(\R^N)$ and $\int_{\R^N} F_2(u_n)\to \int_{\R^N}F_2(u)$.
Then 
\[
E_\alpha\leq J(u)\leq \liminf_{n\to\infty}\frac12\int_{\R^N}(|\nabla u_n|^2+F_1(u_n))-\int_{\R^N}F_2(u)=\lim_{n\to\infty}J(u_n)  =E_\alpha.
\]
Therefore, $E_\alpha$ is attained by $u$ and $|u|$.
Since $E_\alpha$ is attained for each $\alpha$, it is strictly midpoint concave, and thus strictly convex. (i) holds. (ii) follows from Lemma \ref{Lemma:2.1} (iii) and \eqref{GN}.

To see (iii), we assume further  $c_0=0$.  If $f$ admits a zero, then $\lim_{s\to+\infty}f(s)=\lim_{s\to+\infty}F(s)=+\infty$. We can find $u_0\in C_0^\infty(\R^N)$ such that 
$\int_{\R^N}F(u_0)>0$.
Hence, when $c_0=0$, 
\[
  J(u_0(t^{-\frac1N}\cdot))= t^{1-\frac2N}\int_{\R^N}|\nabla u_0|^2-t\int_{\R^N}F(u_0)\to -\infty \mbox{ as } t\to+\infty.
  \]
By $|u_0(t^{-\frac1N}\cdot)|_2^2=t|u_0|_2^2$, we have $E_\alpha\to-\infty$ as $\alpha\to +\infty$. By this, the strong concavity and (ii), $E_\alpha$ has a unique zero in $(0,+\infty)$.

If $f$ is negative in $(0,+\infty)$, then $F(s)=-F_1(s)$ for each $s$.
Hence, $E_\alpha\geq 0$ for each $\alpha$. By the strong concavity and (ii), $E_\alpha$ is strictly increasing.
\end{proof}

\section{Preliminaries for the proof of Theorem \ref{th1.1}}\label{sec2}
By a change of scaling, \eqref{1.1} becomes
\begin{equation}\label{1.1'}
  \begin{cases}
  - \Delta u+V(\e x)u=f(u)+\lambda u,\\
  \int_{\R^N}u^2 =\alpha. 
  \end{cases}
  \end{equation}
  We will solve \eqref{1.1'}
 under the assumptions (F1)--(F5) and (V1)--(V3).
We   assume  
that    $   \Omega \subset B(0,M_0/2)$  for some $M_0>0$, and  without loss of generality, 
\[  \inf_{B(0,M_0) } V=1.\]

For any $O$ satisfies (V3) such that $\mathcal M\subset  O\subset \overline O\subset \Omega$, we can fix $\delta_0 \in(0,1)$ small such that 
$O^{5\delta_0}\subset   \Omega$ and
\[
 \inf_{  O^{3\delta_0 }\setminus O^{\delta_0}}|\nabla V|\geq \nu_0,
 \]
for some $\nu_0>0$, where 
\[
 O^\delta:=\set{x\in \R^N | \text{dist}(x, O)\leq \delta}\quad \mbox{for}\quad  \delta>0.
\]
We now fix 
\begin{equation}\label{mu0}
  \mu_0=\min_{x\in O^{3\delta_0}}V(x).
\end{equation} 

Let $\widetilde V:\mathbb R^N\rightarrow[1,+\infty)$ be a  function such that
\begin{equation}\label{tildev}
\widetilde V(x)=
\begin{cases}
V(x),\ &|x|< M_0;\\
\max\{V(x),|x|^2\}, \ &|x|\geq M_0.
\end{cases}
\end{equation}
We will work on the Hilbert space
\begin{equation}\label{H}
H_\varepsilon:=\Set{u\in H^1(\mathbb R^N) | \int_{\mathbb R^N}\widetilde V(\varepsilon x)u^2\rd x<\infty},
\end{equation}
with inner product 
$$(u,v)_\varepsilon:=\int_{\mathbb R^N}\nabla u \nabla v+ \widetilde V(\varepsilon x) uv,
$$ 
and norm $\|u\|_\varepsilon:=\sqrt{(u,u)_\varepsilon}$. We also denote the norm on dual space by $\|\cdot\|_{H_\e^{-1}}$.
Furthermore, we only prove the existence of $\ell$-peak solutions for $\ell\geq 2$ since the case $\ell=1$ is much simpler.

\subsection{The limit system}
We first study the solution $(\lambda, \boldsymbol{v}):=(\lambda,v_1,v_2,\cdots,v_\ell)\in \mathbb R\times H^1(\mathbb R^N)^\ell$
to  the system
\begin{equation}\label{1.7'}
  \begin{cases}
  -\Delta v_i=f(v_i)+\lambda v_i \  \ \text {in}\ \ \mathbb R^N,\\
  v_i(x)>0,\  \lim_{|x|\to\infty}v_i(x)=0,\quad i=1,2, \cdots, \ell,\\
  \sum_{i=1}^{\ell}|v_i|_2^2=\alpha.
  \end{cases}
  \end{equation}
  Let $\ell^{-1}\alpha\in(0, \alpha_N)$.
 Then  there is a  minimizer $u_0$ for $E_{\ell^{-1}\alpha}$. We may assume $u_0>0$ and $u_0(x)=u_0(|x|)$.
 It is clearly that for some $\lambda\in\R$, $(\lambda,u_0, u_0, \cdots, u_0)\in\R\times H^1(\R^N)^\ell$ is a solution to \eqref{1.7'}.  
 Problem \eqref{1.7'} is related to the following functional 
\[
  \mathbb{J}(\boldsymbol{v}):=\sum_{i=1}^\ell J(v_i)
  \]
  defined on 
  \[
 \mathbb{M}_\alpha:=\set{\boldsymbol{v}=(v_1,\cdots,v_\ell)\in H^1(\R^N)^\ell | \sum_{i=1}^\ell |v_i|_2^2=\alpha}. 
  \]
 We call   $\boldsymbol{v}$ a  critical point to $\mathbb{J}$ on $\mathbb{M}_\alpha$ if $(\lambda, \boldsymbol{v})$ solves \eqref{1.7'} for some $\lambda$.
  Set
\[S_{\ell-1}:=\left\{\boldsymbol s=(s_1,\cdots,s_\ell)\in[0,1]^\ell\, \bigg| \,\sum_{i=1}^\ell s_i=1\right\}.\]
 For each $\boldsymbol{s}\in S_{\ell-1}$, define
 \[\gamma_0(\boldsymbol{s})=  (\sqrt{ {\ell s_1} } u_0, \cdots,    \sqrt{\ell s_\ell}  u_0) \in \mathbb{M}_\alpha.\]
 \begin{proposition}
  For each closed neighborhood $S\subset S_{\ell-1}$ of $\boldsymbol{s}^0:=(\ell^{-1},\cdots,\ell^{-1})$, we   define
  \begin{equation*}\label{eq211}
   \Gamma =\Set{\gamma\in C( S,  \mathbb{M}_\alpha ) |   \gamma=\gamma_0 \ \text{on}\ \partial S }.
 \end{equation*}
There holds 
\begin{equation}\label{eq 4444}
    \mathbb{J}(\gamma_0(\boldsymbol{s})) 
     < \ell E_{\ell^{-1}\alpha} \quad \mbox{if}\quad \boldsymbol{s}\neq \boldsymbol{s}^0.
  \end{equation}
  Moreover,
  \begin{equation}\label{eqmp}
  \ell E_{\ell^{-1}\alpha} = \inf_{\gamma\in\Gamma }\sup_{\boldsymbol s\in  S}  \mathbb{J}(\gamma(\boldsymbol{s})).
  \end{equation}
 \end{proposition}
 \begin{proof}
  By the strict convexity of $F(\sqrt{\cdot})$,  
  for each $\boldsymbol{s}\neq \boldsymbol{s}^0$, we have
 \[ 
 \begin{aligned}
   \mathbb{J}(\gamma_0(\boldsymbol{s}))
   =  &\frac\ell 2\int_{\R^N}|\nabla u_0|^2-\sum_{j=1}^\ell \int_{\R^N} F(\sqrt{ {\ell s_j} } u_0)\\
   = & \frac\ell 2\int_{\R^N}|\nabla u_0|^2-\ell\int_{\R^N} \ell^{-1}\sum_{j=1}^\ell  F(\sqrt{ {\ell s_j}   u_0^2})\\
    < &\frac\ell 2\int_{\R^N}|\nabla u_0|^2-\ell\int_{\R^N}   F\left(\sqrt{ \ell^{-1}\sum_{j=1}^\ell {\ell s_j}   u_0^2}\right)=\ell E_{\ell^{-1}\alpha}.
 \end{aligned}  
 \]
 On the other hand, according to the Brouwer degree theory, for each $\gamma\in \Gamma_\alpha$, there exists $\boldsymbol s\in S $ such that 
 \[
   |(\gamma(\boldsymbol{s}))_i|_2^2=\ell^{-1}\alpha,\quad i=1,\cdots,\ell.
 \]
This implies \eqref{eqmp}.
\end{proof}

By Lemma \ref{Lemma:2.2}, $E_{t\beta}< t E_\beta$ for $t>1$, $t\beta\in(0,\alpha_N)$. So we have
\begin{equation}\label{eq25}
  \ell E_{\ell^{-1}\alpha}<(\ell+1) E_{(\ell+1)^{-1}\alpha}  \text{\ for each\ }  \ell\geq 1\ \mbox{such that}\ \ell^{-1}\alpha\in(0,\alpha_N). 
\end{equation} 
Specially, \eqref{eq25} is true for each $\ell\geq 1$ if $\alpha\in(0,\alpha_N)$.

Let  $\mu_0 \in [1,V_0)$ be the constant   fixed in \eqref{mu0}.
For $\boldsymbol{\mu}=(\mu_1,\cdots,\mu_\ell)\in [  \mu_0, V_0 ]^\ell$  and $\boldsymbol{v}\in\mathbb{M}_\alpha$, we consider the functional 
$\mathbb{J}_{\boldsymbol{\mu}}$ defined by 
\[
  \mathbb{J}_{\boldsymbol{\mu}}(\boldsymbol{v})=\mathbb{J}(\boldsymbol{v})+\sum_{i=1}^\ell\frac12 \mu_i|v_i|_2^2.
\]

Similarly, we say  $\boldsymbol{v}$ is a  critical point of $\mathbb{J}_{\boldsymbol{\mu}}$ on $\mathbb{M}_{\alpha}$ if there is $\lambda\in\R$   such that $(\lambda, \boldsymbol{v})$ solves the following problem:
  \begin{equation}\label{psa}
    \begin{cases}
    -\Delta v_i=f(v_i)-\mu_iv_i+\lambda  v_i \  \ \text {in}\ \ \mathbb R^N,\\
     \lim_{|x|\to\infty}v_i(x)=0, \quad i=1,2,\cdots, \ell,\\
     \sum_{i=1}^\ell |v_i|_2^2= \alpha.
    \end{cases}
    \end{equation}

\begin{lemma}\label{lem1bound} 
  Assume   $ \alpha\in(0, \alpha_N)$. For $\beta\in [\frac12\alpha,\alpha]$, $\boldsymbol{\mu}=(\mu_1,\cdots,\mu_\ell)  \in [ \mu_0, V_0]^\ell $, let  $  \boldsymbol{v} =( v_1,\cdots,v_\ell)$ be a critical point of $\mathbb{J}_{\boldsymbol{\mu}}$ on $\mathbb{M}_{\beta}$ with a corresponding Langrange multiplier $\lambda$. If    $\mathbb{J}(\boldsymbol{v})\leq C_0$  for some  constant $C_0$, then there is $D_1>0$ depending only on $\alpha, C_0, \ell,\mu_0$ such that 
    \[\sum_{i=1}^\ell\|v_i\|_{H^1}  +  |\lambda|\leq D_1  .\]
\end{lemma}
\begin{proof}
   By \eqref{6}, we have 
    \[\begin{aligned}
     C_0  \geq  \sum_{i=1}^\ell J(v_i)
     = & \sum_{i=1}^\ell \left(\frac12  |\nabla v_i|_2^2  - \int_{\R^N} F(v_i)\right)   \\
     \geq &\sum_{i=1}^\ell \left(\frac12  |\nabla v_i|_2^2- \int_{\R^N}(c_0+\tau) |v_i|^{2+\frac4N} -C_\tau \int_{\R^N}|v_i|^2  \right)\\
     \geq &\frac12\sum_{i=1}^\ell\left(1-2 S(N)(c_0+\tau)(|v_i|_2^2)^{\frac2N}\right)|\nabla v_i|_2^2-C_\tau\alpha.
    \end{aligned}
    \]
Here we   fix $\tau>0$   sufficiently small such that 
      \[
        1-2 S(N)(c_0+\tau)(|v_i|_2^2)^{\frac2N}\geq 1-2 S(N)(c_0+\tau)\alpha^{\frac2N}>0.
      \]
  Then    we see that 
  $|\nabla v_i|_2^2$ is bounded by some constant depending only on $\ell$, $N$, $\alpha$ and $C_0$.

  On the other hand, since $(\lambda-\mu_i) |v_i|_2^2=|\nabla v_i|_2^2 -\int_{\R^N} f(v_i)v_i$, $f(v_i)v_i\geq 2F(v_i)$, and $f(v_i)v_i\leq C|v_i|_{2+\frac4N}^{2+\frac4N}$, we have 
  $ |\lambda-\mu_i||v_i|_2^2\leq C$ for some constant $C>0$. 
 Summing up, we have 
  \[
    |\lambda| \alpha  \leq \ell  C +\max_{1\leq i\leq \ell}| \mu_i| \alpha.
  \]
  Then the conclusion follows.
  \end{proof}
  \begin{lemma}\label{L2}
    If $u\geq 0$ satisfies $-\Delta u\leq f(u)+ t u$, then either $u=0$ or $|u|_2^2\geq 1/C_{t}$ for some $C_t>0$ depending only on $t$.
  \end{lemma}
\begin{proof}
 The conclusion follows from   the Gagliardo--Nirenberg inequality:
\begin{equation*}
 |\nabla u|_2^2\leq \int_{\R^N}  f(u)u +\lambda  u^2\leq \int_{\R^N}  (\frac{f(u)}u +t)^+  u^2\leq  C_t  |u|_{2+\frac4N}^{2+\frac4N}\leq C_t C(N) |\nabla u|_2^2 |  u|_2^\frac4N,
\end{equation*}
where $C_t>0$ is a constant depending only on $t$.
\end{proof}
For $ \alpha\in(0, \alpha_N)$, set
 \[
  K_{\alpha}=\Set{ \boldsymbol{v}\in \mathbb{M}_\alpha |
\begin{gathered}
  \boldsymbol{v} \ \mbox{is a critical point of $\mathbb{J}_{\boldsymbol{\mu}}$ on $\mathbb{M}_{ \alpha}$ for some $\boldsymbol{\mu}\in [\mu_0,V_0]^\ell,$} \\
   \mbox{
    $v_i > 0$ and $v_i(0)=\max_{\R^N}v_i$ for $i=1,2,\cdots,\ell$, and }   
      \mathbb{J}_{\boldsymbol{\mu}}(\boldsymbol{v})\leq \ell E_{\ell^{-1}\alpha}+\frac12V_0\alpha\end{gathered} }.
\]
      Clearly, $K_{\alpha}\neq \emptyset$. Moreover, set  
        \begin{equation}\label{rho1}
          \rho_1=\frac12\min\{ { C_{t}^{-1}}, \ell^{-1}\alpha\}
        \end{equation} 
         where $C_t$ is the constant in Lemma \ref{L2} with   $t=D_1+V_0$, and $D_1$ is the constant 
      in Lemma \ref{lem1bound} for some $C_0\in (\ell E_{\ell^{-1}\alpha}, (\ell+1) E_{(\ell+1)^{-1}\alpha})$ (see \eqref{eq25}).
     We have
\begin{lemma}\label{compact}
  There is $C, c>0$ such that 
  for each
  $\boldsymbol{v}\in K_{\alpha}$, there hold $|v_i|^2\geq 2 \rho_1$, and 
  \[
   |v_i(x)|\leq C e^{-c |x|^2}, \quad i=1,\cdots,\ell.
  \]
    Moreover,
  $ K_{\alpha}$ is compact in $H^1(\R^N)^\ell$ and $H_\e^\ell$.
\end{lemma}
\begin{proof}
 The the conclusion follows from Lemma \ref{lem1bound}, Lemma \ref{L2}, and Lemma \ref{decay}.
\end{proof}

In what follows, we write $\boldsymbol{p}=(p_1,\cdot\cdot\cdot,p_{\ell})\in(\R^N)^{\ell}$, 
and set
\begin{equation}\label{xi}
  \xi(\boldsymbol{p})=\min_{1\leq i\neq j\leq \ell}|p_i-p_j|.
\end{equation}

The following estimate is essential to obtain a 
minimax geometry for the functional of \eqref{1.1'}, whose proof   will be given in the Appendix.
\begin{proposition}\label{prop35} Assume (F1)--F(5). Let  $\boldsymbol{v}=(v_1,v_2,\cdots,v_\ell)\in K_{\alpha}$. 
  Then there is $C>0$ such that  for sufficiently large $L$,
\[
J(B\sum_{j=1}^\ell v_j(\cdot -p_j)) + \frac{V_0 }{2}\int_{\R^N}|B\sum_{j=1}^\ell v_j(\cdot -p_j)|^2  \leq \mathbb{J} (\boldsymbol{v})  +\frac{V_0}{2}\alpha- C \xi(\boldsymbol p)e^{-\frac{\sigma \xi(\boldsymbol p)^2}{8}},
\]  
where $\boldsymbol{p}=(p_1,\cdots, p_\ell)\in (\R^N)^\ell$ with
$\xi(\boldsymbol{p})\geq \frac{L}{2}$, and $B=\alpha^{\frac12}|\sum_{j=1}^\ell v_j(\cdot -p_j)|_2^{-1}$.
\end{proposition}

\subsection{Local centers of mass}
We will introduce $\ell$ local centers of mass $(\Upsilon_1(U), \cdots, \Upsilon_\ell(U))$ as in \cite{Byeon-Tanaka}. 
First by Lemma \ref{compact},
 we can find  $R_0>1$ such that for 
each 
$\boldsymbol{U}=(U_1,\cdots,U_\ell)\in  K_{\alpha }$, there holds
\begin{equation}\label{equa 3.10}
\|U_j\|_{L^2(B(0,R_0/2))}>\frac{3}{4}\rho_1,\quad  \|U_j\|_{L^2(\mathbb R^N\setminus B(0,R_0))}<\frac{\rho_1}{8\ell}.
\end{equation}
 Then we have 
\begin{lemma}\label{lemma3.4}
  For   $u\in H^1(\R^N), (y_1,\cdots,y_\ell)\in (\R^N)^\ell, (U_1,\cdots, U_\ell)\in   K_{\alpha }$ such that 
  \[\xi(y_1,\cdots,y_\ell)>12R_0, \quad \|u-\sum_{j=1}^\ell U_j(\cdot - y_j)\|< \frac{\rho_1}{16 },
  \]
   there hold
  \[\int_{B(P, R_0)}   u^2 \geq \frac1 2 \rho_1^2\quad  \mbox{for}\quad  P\in \bigcup_{j=1}^{\ell}\overline B(y_j,R_0/2), \quad \int_{B(P, R_0)}   u^2\leq \frac{1}{16}\rho_1^2 \quad  \mbox{for}\quad  P\notin
  \bigcup_{j=1}^{\ell} B(y_j,2R_0).  
  \]
   
  \end{lemma}
We define 
\begin{equation}
 Z =\Set{u\in H^1(\R^N) | \|u-\sum_{j=1}^\ell U_j(\cdot - y_j)\|< \frac{\rho_1}{16},  \xi(y_1,\cdots,y_\ell)\geq 12R_0,  (U_1,\cdots,U_\ell)\in K_{\alpha } }.
\end{equation}
For $u\in H^1(\R^N) $ and $P\in\mathbb R^N$, we define
\begin{equation}
d(u,P)=\psi\left( \int_{B(P, R_0)}   u^2\right),
\end{equation}
with $\psi\in C_0^\infty([0,\infty), [0,1])$ satisfying
\begin{equation*}
\psi(r)=
\begin{cases}
0 \  \ &r\in[0,\frac{1}{16}\rho_1^2],\\
1 \ \ &r\in[\frac{1}{2}\rho_1^2,\infty).
\end{cases}
\end{equation*}
By   Lemma \ref{lemma3.4}, for any $u\in   Z$ there exist $\ell$ disjoint balls $B_j$ satisfying
\begin{equation*}
  \begin{cases}
  {\rm diam} B_j=5R_0 \  \ &\text {for\ all\ }j\in\{1,2,\cdot\cdot\cdot,\ell\},\\
d(u,\cdot)\not\equiv 0 \  \  &\text{on}\ B_j\ \text {for\ all\ }j\in\{1,2,\cdot\cdot\cdot,\ell\},\\
d(u,\cdot) \equiv 0 \  \  &\text{on}\ \mathbb R^N\setminus \cup_{j=1}^{\ell}B_j.
\end{cases}
\end{equation*}
For   $B_j$, we define
\begin{equation}
\Upsilon_{j}(u)=\frac{\int_{B_j}d(u,P)PdP}{\int_{B_j}d(u,P)dP}\in B_j.
\end{equation}

It is clear that
 $(\Upsilon_1(u), \cdots, \Upsilon_\ell(u))$ is uniquely determined up to permutation and it is
independent of the choice of 
each $B_j$. 
Similar to the argument of \cite{Byeon-Tanaka}, we can assume that
\[
  \Upsilon(u)=(\Upsilon_1(u),\cdots,\Upsilon_\ell(u)),
\]
is continuous up to permutations. 
Note that for a continuous function $\varphi(\boldsymbol{p})$ which is independent of permutation of  $p_i$,  
$\varphi(\Upsilon(u))$ is well defined and continuous.
Moreover, similarly to \cite[Lemma 2.5]{zhangzhang2}, we have the following properties of $\Upsilon$.
 
\begin{lemma}\label{lem3.4}
The following statements hold for true.
\begin{itemize}
  \item[(i)] For $u\in    Z$, we have
  $|\Upsilon_j(u)-y_j|\leq 2R_0$ $(j=1,2,\cdot\cdot\cdot,\ell_0)$ up to permutation.
\item[(ii)] $\Upsilon_j(u)$ is $C^1$ continuous for each $u\in  Z$. Moreover, there exists a constant $D_2>0$   such that
    $$\sup_{u\in  Z}\|\Upsilon_j'(u)\|\leq D_2.$$
    
\item[(iii)] if $u,v\in    Z$ satisfy for some $j\in\{1,\cdot\cdot\cdot,\ell\}$ and $h\in\mathbb R^N$
$$v(x-h)=u(x)\ \ \ \text{in}\ B(\Upsilon_j(u),4R_0),$$
then $\Upsilon_j(v)=\Upsilon_j(u)-h$.
\item[(iv)] $\Upsilon'(u)v=0$ if $\supp v\subset \R^N \setminus \cup_{j=1}^\ell B(\Upsilon_j(u), 4R_0)$.
\end{itemize}
\end{lemma}

\subsection{Penalized functional}\label{subsec 3.3}

We   use notation $ \frac1\varepsilon O^\delta=\set{x\in\R^N|\e x\in O^\delta}$ for $\e,\delta>0$.
Take $\phi\in C_0^\infty(\R^N)$ such that $0\leq \phi\leq 1$ and $|\nabla \phi|\leq 4/\delta_0$ in $\R^N$, $\phi=1$ for $|x|\leq \delta_0/2$, and
$\phi=0$ for $|x|\geq \delta_0$. Set $\phi_\e(x)=\phi(\e x)$.
 For   $L\geq 100R_0$,
set 
\[
Z_{L,\e}=\Set{\sum_{j=1}^\ell (\phi_\e U_j)(\cdot - y_j)| (U_1,\cdots,U_\ell)\in K_{\alpha}, \  \set{y_1, y_2, \cdots, y_\ell}\subset \frac1{\varepsilon} O^{4\delta_0}, \    \xi(y_1,\cdots,y_\ell)\geq L }. 
\]
By compactness of $K_{\alpha}$ and the decay estimate of $U_j\in K_{\alpha}$ (see Lemma \ref{decay}), we know that 
\[
 \| \sum_{j=1}^\ell (\phi_\e U_j)(\cdot - y_j)- \sum_{j=1}^\ell   U_j(\cdot - y_j)\|\leq C e^{-c\e^{-1}},
\]
for some $C, c>0$ independent of the choice of $\e$, $y_j$ and $U_j$, $j=1,\cdots,\ell$.
There for  $ \rho\le\frac{1}{32}\rho_1$, if $u\in H_\e$
is such that 
$\dist_{H_\varepsilon}(u, Z_{L,\e})<  {\rho }$,
then $\Upsilon(u)$ is well defined for small $\e$, since $\|w\|\leq \|w\|_\e$ holds for each $w\in H_\e$.

For
$0<\rho\le\frac{1}{32}\rho_1$, $\delta\in[\delta_0, 3\delta_0]$, set 
 \begin{equation}
 Z(\rho,\delta)= \Set{u\in \mathcal M_\alpha^\varepsilon|\dist_{H_\varepsilon}(u, Z_{L,\e})<  {\rho } ,\quad \max_{1\leq j\leq \ell}\dist(\varepsilon\Upsilon_j(u), O)<\delta     },
 \end{equation}
 where $\mathcal M_\alpha^\e:=\set{u\in H_\e | \int_{\R^N}u^2=\alpha}$.
 Note that $Z(\rho,\delta)$ depends on $L$ and $\e$, we omit them for the sake of brevity.
It is   sufficient  to impose 
\begin{equation}\label{ep}
  0<\varepsilon<\varepsilon_L:=(\frac\delta{4L})^4,
\end{equation} 
so that $Z(\rho,\delta) $ is nonempty when $L$ is   large enough. 
In what follows, we will always assume $\varepsilon\in(0,\varepsilon_L)$ and $L$ is sufficiently large.
\begin{remark}
 Let $\rho<\rho'$ and $\delta<\delta'$. Then,
  \[
  \dist_{H_\varepsilon}(\partial Z(\rho', \delta'), Z(\rho,\delta))>0.  
  \]
  In fact, if $\dist_{H_\varepsilon}(u, Z_{L,\e})=  \rho' $, then $\dist_{H_\varepsilon}(u, Z(\rho,\delta))\geq \rho'-\rho$. 
  If $\dist_{H_\varepsilon}(u, Z_{L,\e})\leq \rho'$ and $ \dist(\e\Upsilon_j(u), O)=\delta'$ for some $j$, then by Lemma \ref{compact} and 
  Lemma \ref{lem3.4} (i), for $\varepsilon$ sufficiently small,
  $\dist_{H_\varepsilon}(u, Z(\rho,\delta))> \|U_j\|_\varepsilon/{2}\geq  {\rho_1}/{2}$.
\end{remark}

As in  \cite{zhang1}, we choose $H(s)\in C_0^\infty ([-3, 3];[0,1])$, with $H(s)=1$ for $|s|\leq 1$, $H'(s)$ is odd
  and $-1\leq H'(s) \leq 0$ for  $s\geq 0$.
Denoting 
\[ \widetilde V_\varepsilon(x)=\widetilde V(\varepsilon x),\quad V_\varepsilon(x)=V(\varepsilon x),\quad \overline V_\varepsilon =V_\varepsilon -\widetilde V_\varepsilon,\]
 we define $\Psi_\varepsilon$ by
\begin{equation}\label{Psi}
\Psi_\varepsilon(u)= \frac{1}{2}\int_{\mathbb R^N} \overline V_\varepsilon (x) H( e^{\e |x|^2} u) u^2 \rd x.
\end{equation}
 
Note that   $\Psi_\varepsilon$ is well-defined on $H_\e$, and if $u(x) \le e^{-\e |x|^2}$ for $x\in \R^N\setminus  B(0, M_0/\e)$, then
\begin{equation}\label{psi'}
   \int_{\R^N}( \nabla u \nabla v+\widetilde V(\varepsilon x)uv)
    +\Psi_\varepsilon'(u)v= \int_{\R^N}( \nabla u \nabla v+ V(\varepsilon x)uv), \quad u, v\in H_\e.
\end{equation}
 We have the following lemma.
\begin{lemma}[Corollary 2.2 of \cite{zhang1}]\label{cor2.3}
 For some $C,c>0$ independent of $\varepsilon$, there holds
\begin{align*}
\sup_{u\in H_\varepsilon}|\Psi_\varepsilon(u)|+
 \sup_{u\in H_\varepsilon}\|\Psi_\varepsilon'(u)\|_{H_\varepsilon^{-1}}\leq Ce^{-c\varepsilon^{-1}},
\end{align*}
where  $\|\cdot\|_{H_\varepsilon^{-1}}$denotes the norm on the dual space of $H_\varepsilon$.
\end{lemma}

Let $\xi(\boldsymbol{p})$ be the function in \eqref{xi} for $\boldsymbol{p}=(p_1,\cdot\cdot\cdot,p_{\ell})\in(\R^N)^{\ell}$. We note that $\boldsymbol{p}\mapsto\min\set{\xi(\boldsymbol{p}),\varepsilon^{-\frac34}}$ is  Lipschitz continuous and independent of permutations of $p_i$.
 
  By the integral convolution with mollifiers, we can find a smooth function  $ \xi_1(\boldsymbol{p})\in C^1((\R^N)^{\ell})$   independent of permutations    of  $p_i$, such that  for some constant $C(N,\ell)>0$ depending only on $N,\ell$,
  \[ \mbox{$ | \xi_1(\boldsymbol{p})-\min\set{\xi(\boldsymbol{p}),\varepsilon^{-\frac34}}|\leq 1$ and $|\nabla\xi_1(\boldsymbol{p})|\leq C(N,\ell) $},\quad \boldsymbol{p}\in(\R^N)^{\ell},\]
Then $u\mapsto\xi_1(\Upsilon(u))$ is well-defined and $C^1$ continuous.
  Take $\chi \in C^\infty(\mathbb R^N;[0,1])$ such that 
  \[\mbox{$\chi=1$ in $  \mathbb R^N\setminus  B(0,  \frac 1  5)$, $\chi =0$ in $ B(0, \frac 1 {10})$ and $|\nabla  {\chi}|\leq 20 $.}\]
  Setting $\chi_{ u}(x)=\Pi_{j=1}^\ell \chi \left(\frac{x-\Upsilon_j(u)}{\xi_1(\Upsilon(u))}\right)$, 
  we note that $\chi_{ u}$ is independent of permutations of $\Upsilon_j(u)$. 
  Define
  $$\Phi_\varepsilon(u)=\left(\xi_1(\Upsilon(u)) \int_{\R^{N} }\chi_{u} u^2 \mathrm dx-1\right)_+^2.
  $$
Then,   by Lemma \ref{lem3.4}, we can check that
\begin{lemma}\label{Phi'}
There is  $C_0>0$ independent of $L, \varepsilon$ such that  for $u\in  Z( \frac{\rho_1}{32}, 3\delta_0)$ and any $v\in H_\varepsilon$,
\[\begin{aligned}
  \Bigg| \Phi_{\varepsilon}'(u)v-&4\Phi_{\varepsilon}(u)^{\frac12} \xi_1(\Upsilon(u))\int_{\R^{N} }\chi_u uv     \Bigg| 
  \leq  C_0\Phi_{\varepsilon}(u)^{\frac12}  \|v\|_\e  \int_{ \R^N\setminus \cup_{j=1}^\ell B(\Upsilon_j(u),\frac{1}{10}\xi_1( \Upsilon(u)))}  u^2.
\end{aligned}\]
Moreover, if $\supp v\subset \R^N \setminus \cup_{j=1}^\ell B(\Upsilon_j(u), 4R_0)$, then 
\[ 
   \Phi_{\varepsilon}'(u)v=4\Phi_{\varepsilon}(u)^{\frac12} \xi_1(\Upsilon(u))\int_{\R^{N} }\chi_u uv.
 \]
\end{lemma}
We also modified the nonlinearity term. Recalling the definition of $f_1$ and $f_2$, we define odd function 
$$f_{2,K}(t):=\min\{f_2(t),f_2(K)\}\ \ \text{for any }
K>0\ \text{and }t\geq 0.$$  
Set $f_{K}(t):=-f_1(t)+f_{2,K}(t)$,
$F_{2,K}(t):=\int_0^{t} f_{2,K}(s) \rd s$ and 
$F_{K}(t):=-F_1(t)+F_{2,K}(t)$. Then necessarily, 
$$f_K(t)=\min\{f(t),f(K)\}.$$
Define the functional:
\begin{equation}\label{2.3}
\Gamma_{\varepsilon,K}(u)=\frac{1}{2}\int_{\R^N}(|\nabla u|^2+\widetilde V(\varepsilon x)u^2)
- \int_{\mathbb R^N}F_K(u)+\Phi_{ \varepsilon}(u) +\Psi_\varepsilon(u),\quad u\in  Z( \frac{\rho_1}{32}, 3\delta_0).
\end{equation}
We note that by \cite[Lemma 2.3]{zhang1}, $\Gamma_{\varepsilon,K}$ is well-defined and is of  class $C^1$ on $Z_L(3\delta_0, \rho_1)$.
For $u\in H_\e$, we also set \[
G(u)=\frac12\int_{\R^N}u^2 \rd x. 
\]
\begin{lemma}\label{lem3.12}
  If $u\in  Z( \frac{\rho_1}{32}, 3\delta_0)$   satisfies
    $\Gamma_{\varepsilon,K}(u) \le (\ell+1) E_{(\ell+1)^{-1}\alpha}+\frac12V_0\alpha$,   
   then   the following quantities are bounded by a constant $C_0>0$ independent of $\varepsilon$, $L$ or $K$:
   \begin{equation*}
    \begin{gathered}
     \|u\|_\e,\   \Phi_{ \varepsilon}(u), \ \int_{\R^N}  f_1(u)u,\ \int_{\R^N}  f_{2,K}(u)u,\ \int_{\R^N} F_1(u),\ \int_{\R^N} F_{2,K}(u), \ 
         \xi_1(\Upsilon(u))\int_{\R^{N} }\chi_u  u^2 .
    \end{gathered}
   \end{equation*}
       If we assume additionally that $\|\Gamma_{\e,K}' (u )-\lambda G'(u)\|_{H_\e^{-1}}\leq 1$, then making $C_0$ larger if necessary, it holds $|\lambda|\leq C_0$.
\end{lemma}
\begin{proof}
  Clearly,
  $\|u\|_{\varepsilon }\leq C$ for some $C>0$ independent of  $L, \varepsilon, K$. Hence, by $|f_{2,K}(t)|\leq |f_2(t)|\leq C|t|^{1+\frac4N}$, we have
  $\|f_{2,K}(u)u\|_{L^1(\R^N)}+\|F_{2,K}(u)\|_{L^1(\R^N)}\leq C$.
  Then,   we have 
   \[
    \Phi_{ \varepsilon}(u) +\int_{\R^N}F_1(u)
    \leq \Gamma_{\varepsilon, K}(u)-\frac12 \|u\|_\e^2 +\int_{\R^N}F_{2,K}(u) -\Psi_\varepsilon(u)\leq C,
   \quad
   \xi_1(\Upsilon(u))\int_{\R^{N} }\chi_u  u^2 \leq  \Phi_{\e}(u)^\frac12+ 1 \leq C.
\]
By Lemma \ref{F12} (iii), there holds
\[\frac12\int_{\R^N}f_1(u)u\leq \int_{\R^N}F_1(u)  \leq C.\]
Thus the first conclusion follows.
To show $|\lambda|\leq C_0$, we see that
   \begin{gather*}
     |\Gamma_{\e,K}'(u)u|  \leq \|u\|_\e^2  + \int_{\R^N} |f_K(u)u|   +|\Phi_{\e}'(u) u|+|\Psi_\e'(u)u|\leq C, \notag\\
  |\lambda|\leq \alpha^{-1}(\|u\|_\e+  |\Gamma_{\e,K}'(u) u|)\leq C_0. 
\end{gather*}
which complete the proof.
\end{proof}

\begin{remark}\label{re3.12}
 Let $u\in H^1(\R^N)$ weakly solves  the following equation
\begin{equation*}
-\Delta|u|\leq f_{2}(|u|)+|\lambda u|\quad \text{in } B(x,1),
  \end{equation*}
 with $\|u\|\leq C_0$ and $|\lambda|\leq C_0$, where $x\in\R^N$ is arbitrary and $C_0$ is the constant in Lemma \ref{lem3.12}.
Then by the fact
$  f_2(t) \leq Ct^{1+\frac4N}$\quad for any $t\geq 0,$
and the subsolution estimates \cite{Trudinger1997}, it follows $\|u\|_{L^\infty(B(x,1/2))}\leq K_0$ for some constant $K_0>0$.
Making $K_0$ larger if necessary, then
\begin{equation}\label{e39}
2u_0\leq K_0.
\end{equation}
From now on, we fix $K=K_0$, and denote
$\Gamma_\e(u):=\Gamma_{\varepsilon,K_0}(u)$.
Moreover, we set 
$\bar f_2=f_{2,K_0}$,
$\bar F_2=F_{2,K_0}$, $\bar f=f_{K_0}$ and $\bar F=F_{K_0}$, and hence there always hold
$\bar f_2(t)\leq f_2(K_0)$ and $\bar f(t)\leq f(t)$ for $t\geq 0$.
\end{remark}

\subsection{A prior decay estimate}
The following lemma is useful to get a priori decay estimate.
 \begin{lemma}\label{lemma 2.9}
  Let $\theta>1$, $b\geq 0$, $R_1, R>0$ be such that $R>R_1+1$. 
  Assume $Q(r)$ is a nonincreasing function in $[R_1, R]$  satisfying
  \[
    Q(r)\leq \theta^{-1} Q(r-1) +b \quad \mbox{for } r\in [R_1+1, R].
  \]
  Then 
  \[
  Q(R)\leq \theta^{R_1+1}Q(R_1) e^{-R\ln\theta}+\frac{\theta b}{ \theta-1}.  
  \]
 \end{lemma}
 \begin{proof}
  By the assumptions, we can get the conclusion from 
  \[
  (Q(R)- \frac{\theta b}{ \theta-1})^+ \leq \theta^{-1} (Q(R-1)- \frac{\theta b}{ \theta-1})^+\leq \theta^{-\lfloor R-R_1\rfloor} (Q(R_1)- \frac{\theta b}{ \theta-1}). \qedhere
  \]
  \end{proof}
\begin{proposition}\label{lembound} There is $\rho_0\in (0, \rho_1/96)$ and  $L_1\geq 100 R_0$ such that the following statements hold for each  $L\geq L_1$ and $\e \in (0, \e_L)$.
 If  $u\in  Z (3\rho_0, 3\delta_0)$ and  $\lambda\in\R$  satisfy
 \begin{gather*}
  \Gamma_{\varepsilon}(u) \le (\ell+1) E_{(\ell+1)^{-1}\alpha}+\frac12V_0\alpha,\\
   \|\Gamma_{\e}' (u )-\lambda G'(u)\|_{H_\e^{-1}}\leq b_\e\quad \mbox{for some   $b_\e\geq 0$},
 \end{gather*}
   then there is $C,c>0$ independent of $\e, L, b_\e$ such that $|\lambda |\leq C(1+b_\e)$ 
   and
 for each $R\geq 8R_0$,  
 \[\int_{\R^N\setminus \cup_{j=1}^\ell B(\Upsilon_j(u), R)    }\left(|\nabla u|^2+  u^2  \right) \leq C(b_\e+e^{-cR}+e^{-\frac{c}{\varepsilon}}).\]  
\end{proposition}

\begin{proof}
By Lemma \ref{lem3.12},  we have
    \begin{gather}
     |\Gamma_{\e}'(u)u|  \leq \|u\|_\e^2  + \int_{\R^N} |\bar f(u)u|   +|\Phi_{\e}'(u) u|+|\Psi_\e'(u)u|\leq C, \notag\\
  |\lambda|\leq \alpha^{-1}(b_\e\|u\|_\e+  |\Gamma_{\e}'(u) u|)\leq C(1+b_\e).  \label{eqlamdasd}
\end{gather}

  First note that,   by Lemma \ref{lem3.4} and the compactness of $K_\alpha$, for each given $\rho_0\in (0,  \rho_1/96)$,  there is $R_1>4R_0$ such that 
  \begin{equation}\label{eq2.15}
    \sup_{u\in  Z (3\rho_0, 3\delta_0)}  \int_{\R^N\setminus  B(\Upsilon_j(u), R_1)    }  \left(|\nabla u|^2+  u^2  \right)  \leq  10 \rho_0^2.
  \end{equation}

We fix $L_1> R_1+1$ and consider $L\geq L_1$.
  For  $R\in[R_1+1, L]$ and $r \in [R_1+1, R] $, we take $\psi_r \in C^1(\mathbb R^N,[0,1])$
  such that $| \nabla \psi_r |\leq 2$ and
 \[
 \psi_r (x)=
 \left\{\begin{aligned}
 &0 & &{\rm if} & &x\in \cup_{j=1}^\ell B(\Upsilon_j (u),   r-1),\\
 &1 & &{\rm if} & &x\in \R^N\setminus \cup_{j=1}^\ell  B(\Upsilon_j (u),   r),
 \end{aligned}\right.
 \]
 Since $u\in Z (\rho_1/32, 3\delta_0)$, there is $C>0$ independent of $\e$, $L$, $r$ and $u$ such that
 \[
 \|\psi_r u\|_\varepsilon \leq C.
 \]
 We have 
\[
\begin{aligned}
  \Gamma_{\e}' (u )(\psi_r u)-\lambda \int_{\R^N} \psi_r u^2\leq b_\e\|\psi_r u\|_\varepsilon\leq C b_\e.
\end{aligned}  
\]
 By Lemma \ref{Phi'} and $\supp (\psi_r u) \subset  \R^N\setminus \cup_{j=1}^\ell B(\Upsilon_j(u), 4R_0)$, we have 
 \[
 \Phi_\e'(u)(\psi_r u)=4\Phi_{\varepsilon}(u)^{\frac12} \xi_1(\Upsilon(u))\int_{\R^{N} }\chi_u \psi_ru^2\geq 0. 
 \]
  Together with Lemma \ref{cor2.3}, 
 we have
 \begin{equation}\label{eq4.2}\begin{aligned}
   C b_\e\geq& 
   \int_{\R^N}\psi_r (|\nabla u|^2+ \widetilde{V}_\varepsilon u^2-\bar f(u)u -\lambda u^2)  
   +\int_{\R^N} u\nabla \psi_r \nabla u  + O(e^{-\frac{c}\varepsilon})
 \\
   \geq & \int_{\R^N}\psi_r (|\nabla u|^2+   u^2 -f(u)u -\lambda u^2) 
    -  \int_{\supp |\nabla \psi_r|} ( |\nabla u|^2 +  u^2) + O(e^{-\frac{c}\varepsilon}).
 \end{aligned}
 \end{equation}
 By \eqref{eqlamdasd} and (F2),
 \[
  f(u)  +\lambda u^2\leq (\frac{f(u)}{u} +C)^+u^2+Cb_\e u^2\leq C|u|^{2+\frac{4}{N}}+Cb_\e u^2.
 \]
Setting 
  \[
  Q(r)=\int_{\R^N\setminus \cup_{j=1}^\ell B(\Upsilon_j(u), r)} |\nabla u|^2+    u^2,
  \]
  we conclude from \eqref{eq4.2} and  the Sobolev inequality that
 \[\begin{aligned}
  C(b_\e + e^{-\frac{c}{\e}})\geq& 2Q(r)   - Q(r-1) -
  C \int_{\mathbb R^N}\psi_r |u|^{2+\frac{4}{N}}\\
    \geq &  2Q(r)   - Q(r-1) -C_NC (Q(r-1))^{2+\frac{4}{N}},
 \end{aligned}
 \]
 where $ C_{ N}>0$ is a constant  depending only on   $N$.  By \eqref{eq2.15}, $Q(r-1)\leq \sqrt{10} \rho_0$.
 Taking $\rho_0>0$ small such that
 $C_NC(\sqrt{10}\rho_0)^{ 1+4/N}<1$, we   complete the proof by Lemma \ref{lemma 2.9}.
\end{proof}
A direct corollary is that, when $L\geq L_1$   and $\e$ is sufficiently small,
 $\Phi_{ \e}(u_\e)$ disappear  for a critical point $u_\e$ of $\Gamma_{\varepsilon}$ on $\mathcal M_\alpha^\e$.
 
 In what follows, we denote by $\Gamma_{\varepsilon}|_{\mathcal M_\alpha^\e}'(u)$ the 
 derivative of $\Gamma_{\varepsilon}$ restricted on $\mathcal M_\alpha^\e$ at $u$.
 We denote by $T_u\mathcal M_\alpha^\e:=\set{v\in H_\e | \int_{\R^N}vu=0}$ the tangent space of $\mathcal M_\alpha^\e$ at $u\in \mathcal M_\alpha^\e$.
 We also denote by $\|\cdot\|_*$ the norm of the cotangent space. Note that
 \[
 \| \Gamma_{\varepsilon}|_{\mathcal M_\alpha^\e}'(u)\|_*=\inf_{\lambda\in\R}
 \|\Gamma_{\varepsilon}'(u)-\lambda G'(u)\|_{H_\e^{-1}}.
 \]
\begin{corollary}\label{lem2.8}
  For 
   $  u_\e\in  Z (3\rho_0, 3\delta_0)$  with $ \limsup_{\varepsilon\to0} \Gamma_{\varepsilon}\left(u_{\e}\right) \le  \ell E_{(\ell+1)^{-1}\alpha}+\frac12V_0\alpha$,     
    if\[\xi_1(\Upsilon(u_\e))\|\Gamma_{\varepsilon}|_{\mathcal M_\alpha^\e}^{\prime}\left(u_\e\right)\|_*  \to 0\  \mbox{ as $\e\to 0$,}\]
     then 
  $\Phi_{ \e}(u_\e)=0$ and $\Phi_{\e}'(u_\e)=0$
     for       $L\geq L_1$ and  small $\e$.
\end{corollary}
By the compact embedding from $H_\varepsilon$ to $L^q(\R^N)$ for $q\in (\frac{2N}{N+2}, 2^*)$ (\cite[Lemma 2.3]{zhang1}), it is standard to show the Palais--Smale condition for fixed $\e$, i.e.,
\begin{proposition}\label{pro2.5}
For    $L\geq L_1$ , if $\{u_n\}\subset  Z (3\rho_0, 3\delta_0)$ is such that $\lim_{n\to\infty}\Gamma_{\varepsilon}(u_n) \le  \ell E_{(\ell+1)^{-1}\alpha}+\frac12V_0\alpha$    and  $\|\Gamma_{\e}|_{\mathcal M^\e_\alpha}'(u_n)\|_{*}\to 0$ as $n\to+\infty$,
then $\{u_n\}$ has a convergent subsequence.
\end{proposition}

We can also show the following $\e$-dependent concentration compactness   result.
\begin{proposition}\label{pro2.6}
   For $L\geq L_1$, suppose
   $\varepsilon_n\to0, u_n\in  Z (3\rho_0, 3\delta_0)$   satisfying
\begin{equation}\label{13}
  \limsup_{n\to\infty}\Gamma_{\varepsilon_{n}} (u_{n} ) \leq \ell E_{\ell^{-1}\alpha}+\frac12V_0\alpha,\quad
  \lim_{n\to\infty} \|\Gamma_{\varepsilon_{n}}|_{\mathcal M_\alpha^\e}^{\prime} (u_{n} ) \|_{*}  = 0.
\end{equation}
Then there exist
 $\boldsymbol{U}\in K_{\alpha}$ and 
 $\left(z_{n,j} \right) \subset \R^N, j= 1,2, \cdots, \ell$ such that as $n \to \infty$ (after extracting a subsequence if necessary)
 \begin{enumerate}
      \item $ | z_{n,j}-\Upsilon_j(u_n)|\leq 2  R_0$ for $j=1,2, \cdots, \ell$,
\item $|z_{n,i}-z_{n,j}| \to \infty \ $ for $1 \leq i  < j \leq \ell$,
\item $\|u_{n}- \sum_{j=1}^\ell (\phi_{\e_n} U_j) (\cdot-z_{n,j} )\|_{ {\varepsilon_n}}\to 0$, 
where $U_j$ is the $j$-th component of $\boldsymbol{U}$.
 \end{enumerate}
\end{proposition}

\begin{proof}
 Let  $\varepsilon_n, u_n$
  satisfy 
\eqref{13}. 
By the compactness of $K_{\alpha}$, we can write
\begin{equation}\label{equa 4.1}
u_n =\sum_{j=1}^{\ell}(\phi_{\e_n}\tilde U_j)  (\cdot-y_n^j)+w_n,\quad \|w_n\|_{ \varepsilon_n}\leq 3\rho_0,\quad \varepsilon_n\Upsilon_j(u_n)\in O^{3\delta_0}, 
\quad  \xi(y_n^1,\cdots,y_n^\ell)\geq L,
\end{equation}
 where $(\tilde U_1,\cdots, \tilde U_\ell) \in K_{\alpha}$. By Lemma \ref{lem3.4} (i), 
  $\dist(\varepsilon_n y_n^j, O^{3\delta_0}) \leq 2R_0\varepsilon_n\to 0.$ 
The second equation in \eqref{13} implies that there is $\lambda_n\in \R$ such that
 \begin{equation}\label{35'}
   \|\Gamma_{\e_n}'(u_n)-\lambda_n G'(u_n)\|_{H^{-1}_\e}\to 0.
 \end{equation}
Hence, by Lemma \ref{lem3.12} and Proposition \ref{lembound}, for constant $C_0>0$ in Lemma \ref{lem3.12} and some $C, c>0$ independent of   $L, n$, there hold 
\begin{gather}
  \|u_n\|_{\varepsilon_n},\ \int_{\R^N} f_1(u_n)u_n ,\ \int_{\R^N}F_1(u_n),\ \int_{\R^N} \bar f_2(u_n)u_n\ \int_{\R^N}\bar F_2(u_n), \ |\lambda_n |\leq C_0,\label{e45}\\
   \int_{\R^N\setminus \cup_{j=1}^\ell B(\Upsilon_j(u_n),\frac{1}{10}\xi_1( \Upsilon(u_n))) } \left(|\nabla u_n|^2+  u_n^2 \right) \mathrm dx\leq C e^{-c\xi_1(\Upsilon(u_n))}+o_n(1).\label{eq39}
\end{gather}
By \eqref{eq39} and $\xi_1(\Upsilon(u_n)) \geq \xi(\Upsilon(u_n))-1 \geq L-4R_0-1$, we can assume $L_1$ is so large   that
\begin{equation}\label{eq40}
  \xi_1(\Upsilon(u_n)) \int_{\R^{N} }\chi_{u_n} u_n^2  \leq CLe^{-cL} +o_n(1) \xi_1(\Upsilon(u_n))\leq \frac12 +o_n(1) \xi_1(\Upsilon(u_n)).
\end{equation}
Up to a subsequence, we assume for $j=1,\cdots,\ell$, $\lambda_n\to \lambda_0$,
$\varepsilon_n y_n^j\to y^j\in \overline{O^{3\delta_0}}$ and $ u_n(\cdot+y_n^j)\rightharpoonup W_j\neq 0$ in $H^1(\R^N)$.
 Note that by \eqref{eq40}, if $\xi(\Upsilon(u_n))$ is bounded, then $\Phi_{ \e_n}'(u_n)=0$ for every large $n$.
 So 
in either case that $\xi(\Upsilon(u_n))$ is bounded or $\xi(\Upsilon(u_n))\to+\infty$, we can verify that 
$W_j$ satisfies
$-\Delta u=\bar f(u)+(\lambda_0-V(y^j)) u\ \text{in}\  \R^N.
$ Applying Kato's lemma, we deduce that $|W_j|$ satisfies $$-\Delta v\leq -f_1(v)+\bar f_2(v)+(\lambda_0-V(y^j)) v\leq  f_2(v)+\lambda_0v.$$ 
By this, \eqref{e45} and Remark \ref{re3.12}, we get $|W_j|\leq K_0$, and hence $\bar f(W_j)=f(W_j)$. Thus $W_j$ satisfies
$$
-\Delta u= f(u)+(\lambda_0-V(y^j)) u\quad \text{in}\  \R^N.
$$

\noindent\textbf{Step 1.} We  show that $\xi(\Upsilon(u_n))\to+\infty$  as $n\to+\infty$.

Since \[\sum_{j=1}^\ell|W_j|_2^2\geq \liminf_{n\to +\infty}\|u_n\|^2_{L^2(\cup_{j=1}^\ell B(  y_n^j, 4R_0))}\geq \sum_{j=1}^\ell(\|\tilde U_j\|_{L^2(B(0, 4R_0))}-3\rho_0)^2\geq \frac12\alpha,\]
 we obtain that by Lemma \ref{lem1bound}, 
 $|\lambda_0|\leq D_1$. Hence,  $W_j^-$ satisfies  
\[
-\Delta W_j^- \leq f(W_j^-) +D_1   W_j^-.
\]
Then it follows from $|W_j^-|_{L^2}\leq\limsup_{n\to+\infty} \|w_n\|_{\e_n} \leq 3\rho_0<\rho_1/16$, Lemma \ref{L2} and the definition of $\rho_1$ in \eqref{rho1}, that 
 $W_j^-=0$. Hence, by Lemma \ref{decay}, $W_j$ is  positive and radially symmetric about some point.

Up to a subsequence, we may assume that the index set 
$\{1,\cdots,\ell_1\}$ with $\ell_1\geq 2$ satisfies $\lim_{n\to\infty}|y_n^i-y_n^j|<+\infty$ for $1\leq i<j\leq \ell_1$ and 
$\lim_{n\to\infty}|y_n^i-y_n^k|=+\infty$ for $1\leq i\leq \ell_1$ and $k\geq \ell_1+1$.
 Assume $y_n^j-y_n^1\to z_j\in \R^N$ for $j=2,\cdots,\ell_1$.
Then we have
\[
  \|W_1\|_{L^2(B(0, R_0))}\geq \liminf_{n\to\infty}\|u_n\|_{L^2(B(y_n^1, R_0))}  \geq \|\tilde U_1\|_{L^2(B(0, R_0))}-\sum_{j=2}^\ell \|\tilde U_j\|_{L^2(\R^N\setminus B(0, R_0))}-3\rho_0>  \frac{\rho_1}2.
\]
 Similarly, \[
  \|W_1\|_{L^2(B(z_j, R_0))}>\frac{\rho_1}2,\ \ j=2,\cdots,\ell_1.
 \]
 Setting $z_1=0$, by \eqref{equa 3.10}
 \[
  \|W_1\|_{L^2(\R^N\setminus(\cup_{j=1}^{\ell_1} B(z_j, R_0))}
  \leq\sum_{j=1}^{\ell_1}\|\tilde U_j\|_{L^2(\R^N\setminus  B(z_j, R_0) )}+3\rho_0
  \leq \frac{\ell_1\rho_1}{8\ell}+3\rho_0<\frac{\rho_1}{4}.
 \]
 Then $W_1$ can not be radially symmetric about any point,
which is a contradiction.

\noindent\textbf{Step 2.}
Setting $v_n:=u_{n}- \sum_{j=1}^\ell (\phi_{\e_n} W_j) (\cdot-y_{n}^{j} )$, we show $|v_n|_p\to0$ for $p\in(2,2^*)$.
 
Otherwise, by Lions' Lemma,
there is $y_n$ such that $|y_n-y_n^j|\to\infty$ for each $j=1,\cdots, \ell$
and 
 \begin{equation}\label{eq49}
  \limsup_{n\to\infty}\|u_n(\cdot+y_n)\|_{L^2(B(0,1))}>0.
\end{equation}
By Lemma \ref{Phi'}, \eqref{eq39}, and $\xi_1(\Upsilon(u_n))\to +\infty$, there holds
\[
  \Phi_{\varepsilon_n}'(u_n)v-4\Phi_{\varepsilon_n}(u_n)^{\frac12} \xi_1(\Upsilon(u_n))\int_{\R^{N} }\chi_{u_n} u_nv =o_n(1)\|v\|_{\e_n},\ \ v\in H_{\varepsilon_n}.
  \]
Set
\[R_n:=\frac12\min_{1\leq j\leq \ell}\{|y_n-y_n^j|\},\]
and let   $\eta_n\in C^\infty_0(\R^N,[0,1])$  be such that  $\eta_n=1$ in $B(y_n,1)$, $\eta_n=0$ in $\R^N\setminus B(y_n,R_n)$ and $|\nabla \eta_n|\leq 2/R_n$.
We have 
$$\begin{aligned}
  o_n(1)=&\Gamma_{\varepsilon_n}'(u_n)(\eta^2_n u_n)-\int_{\R^N} \lambda_n\eta^2_n u_n^2 \mathrm d x\\
=&\int_{\R^N}\left(\nabla u_n\nabla (\eta_n^2 u_n) +\widetilde V_\varepsilon \eta_n^2 u_n^2-\eta_n^2 \bar f(u_n) u_n-\lambda_0\eta^2_n u_n^2\right) \mathrm d x\\
&+4\Phi_{\varepsilon_n}(u_n)^{\frac12}\xi_1(\Upsilon(u_n)) \int_{\R^{N} }\chi_{u_n}       \eta_n^2 u_n^2 \mathrm dx+o_n(1)\\
\geq &\int_{\R^N}\left(|\nabla (\eta_nu_n)|^2 +  \eta_n^2 u_n^2-|\nabla\eta_n|^2u_n^2-(\frac{f(u_n)}{u_n} +\lambda_0 )^+\eta_n^2 u_n^2\right) \mathrm d x 
 +o_n(1)\\
\geq& \int_{\R^N}\left(|\nabla (\eta_nu_n)|^2 +  \eta_n^2 u_n^2- \frac 4{R_n^2}u_n^2-C_N\eta_n^2 u_n^{2+\frac4N}\right) \mathrm d x 
 +o_n(1)\\
= & \int_{\R^N}\left(|\nabla (\eta_nu_n)|^2 +  \eta_n^2 u_n^2- C_N\eta_n^2 u_n^{2+\frac4N}\right) \mathrm d x +o_n(1).
\end{aligned}
$$
 By \eqref{equa 4.1}, $\|u_n\|_{H^1( B(y_n, R_n))}\leq \sum_{j=1}^{\ell}\|\tilde U_j  (\cdot-y_n^j)\|_{H^1(  B(y_n , R_n))}+3\rho_0\leq 4\rho_0 $ for large $n$.
Therefore,
$$\begin{aligned}
  \|\eta_n u_n\|^2\leq C \int_{\R^N}\eta_n^2 u_n^{2+\frac4N}+ o_n(1)&\leq C_{N}\|\eta_n u_n\|^2\|u_n\|_{H^1(  B(y_n , R_n))}^{\frac4N}+ o(1)\\
  &\leq C_{N}4^{\frac4N}\rho_0^{\frac4N}\|\eta_n u_n\|^2+ o_n(1),
\end{aligned}
$$
where $C_{N}$ is a constant. Decreasing $\rho_0$ if necessary, there holds  $C_{N}4^{\frac4N}\rho_0^{\frac4N} <1$. 
Therefore,
$$\limsup_{n\to\infty}\|u_n(\cdot+y_n)\|_{L^2(B(0,1))}\leq\limsup_{n\to+\infty}\|\eta_n u_n\|^2=0,
$$
which is a contradiction to \eqref{eq49}.

\noindent{\bf Step 3.} $\| v_n\|_{\e_n}\to 0.$

We test 
\eqref{35'} by $v_n$  to get
\begin{equation}\label{38'}
   \begin{aligned}
  I+II :=&  \int_{\R^N}\left(\nabla u_n\nabla v_n +\widetilde V_\varepsilon   u_n v_n-  \bar f(u_n) v_n -\lambda_n u_n v_n\right) \\
  &+4\Phi_{ \varepsilon_n}(u_n)^{\frac12}\xi_1(\Upsilon(u_n))\int_{\R^{N} }\chi_{u_n} u_n v_n \mathrm dx
 =o_n(1).
  \end{aligned}
\end{equation}
By Lemma \ref{decay},
$$\begin{aligned}
   \int_{\R^{N} }\chi_{u_n}  u_n v_n  \mathrm dx 
  =&
\int_{\R^{N} }\chi_{u_n} |u_n|^2-\chi_{u_n} u_n \sum_{j=1}^\ell (\phi_{\e_n} W_j) (\cdot-y_{n}^{j} )dx\\
\geq& -C e^{-\xi_1(\Upsilon(u_n))}.
\end{aligned}
$$
Hence, $II \geq -o_n(1)$, which implies $I\leq o_n(1)$.
Then  we have
$$\begin{aligned}
  \|v_n\|^2_{ {\varepsilon_n}} 
  &
  =I-\int_{\R^N}\sum_{j=1}^\ell\left(\nabla  (\phi_{\e_n}  W_j )(\cdot-y_{n}^{j})\nabla v_n +\widetilde V_\varepsilon     (\phi_{\e_n}  W_j ) (\cdot-y_{n}^{j})v_n\right)+\int_{\R^N}(\bar f(u_n)+\lambda_n u_n) v_n.
\end{aligned}
$$
We have, by $v_n(\cdot+y_n^j) \rightharpoonup 0$ in $H^1(\R^N)$  and the decay property of $W_j$,
 \[
  \int_{\R^N}  \nabla    (\phi_{\e_n}W_j)  (\cdot-y_{n}^j )\nabla v_n =\int_{\R^N}  \nabla    W_j \nabla (v_n(\cdot+y_n) ) +o_n(1)=o_n(1),
 \]
 \[
  \int_{\R^N}|V_\varepsilon    (\phi_{\e_n}  W_j)  (\cdot-y_{n}^j )v_n|\leq \int_{\R^N} V_0 |W_j v_n(\cdot+y_n) |  =o_n(1).
 \]
 Then 
 \[\|v_n\|^2_{ {\varepsilon_n}} \leq o_n(1)+\int_{\R^N}( \frac{\bar f (u_n)}{u_n}+\lambda_n )   u_n v_n.\]
 Note that $\bar f(t)=f(t)$ for $|t|\leq K_0$
  and $\bar f(t)\leq f(t)$ for $t\geq 0$.
By (F5), there is $\delta>0$  such that $|\bar f(t)/t|\leq t^{-\frac12}$ for $|t|\leq \delta$. By (F2), making $\delta$ smaller if necessary,
$\bar f(t)/t+\lambda_n\leq 0$ for $|t|\leq\delta$. By (F3),
$|\bar f(t)/t+\lambda_n|\leq C|t|^{\frac 4N}$ for $|t|\geq\delta$. So we have 
\begin{equation*}
  \int_{\{|u_n|\geq \delta\}}( \frac{\bar f (u_n)}{u_n}+\lambda_n )   u_n v_n\leq C \int_{\R^N}|u_n|^{1+\frac4N}|v_n|\to 0.
\end{equation*}
On the other hand, for any $R>0$, setting  
$B_R=\cup_{j=1}^\ell B(y_n^j,R)$, we have that
\[
   \int_{\{|u_n|\leq \delta\}\cap B_R} ( \frac{\bar f (u_n)}{u_n}+\lambda_n )  u_n v_n \leq  \int_{B_R}    |u_n|^\frac12 |v_n|\leq  R^{N(\frac34-\frac{N}{2N+4})} |u_n|_2^\frac12|v_n|_{2+\frac4N}\to0,
\]
\[\begin{aligned}
   \int_{\{|u_n|\leq \delta\}\setminus B_R} ( \frac{\bar f (u_n)}{u_n}+\lambda_n )  u_n v_n \leq &   \int_{\{|u_n|\leq \delta\}\setminus B_R} | \frac{\bar f (u_n)}{u_n}+\lambda_n |  |u_n| \sum_{j=1}^\ell W_j (\cdot-y_{n}^{j} )\\
 \leq & C \int_{\mathbb R^N\setminus B_R}  |u_n|^\frac12 \sum_{j=1}^\ell W_j (\cdot-y_{n}^{j} )\leq C  |u_n|_2^\frac12 e^{-R}.
\end{aligned}
\]
Hence, there holds $\lim_{n\to\infty} \|v_n\|^2_{ {\varepsilon_n}}=0$.

\noindent\textbf{Step 4.} Completion of the proof.
Let $z^j$ be the unique maximum point of $W_j$, 
we set $\boldsymbol{U}=(U_1,\cdots,U_\ell)=(W_1(\cdot+z^1),\cdots,W_\ell(\cdot+z^\ell))\in H^1_r(\R^N)^\ell$.
Since \[\int_{\R^N\setminus B(0, 2R_0)}W_j^2=\lim_{n\to\infty}\int_{B(y^j_n, \frac13\Upsilon(u_n)\setminus B(y^j_n, 2R_0)}u_n^2\leq \frac{\rho_1^2}{16\ell^2}, \] 
we have $|z^j|\leq 2R_0$.
By Step 3 and similarly to Lemma \ref{Brezislieb} (ii),
 we have 
\[
\lim_{n\to+\infty}\int_{\R^N}\bar F(u_n)=\sum_{j=1}^\ell\int_{\R^N} \bar F(W_j)
=\sum_{j=1}^\ell\int_{\R^N}  F(W_j)
=\sum_{j=1}^\ell\int_{\R^N}  F(U_j).
\]
Therefore, for $\boldsymbol{\mu}=(V(y^1),\cdots,V(y^\ell))\in [\mu_0,V_0]^\ell$,
$$\begin{aligned}
  \mathbb{J}_{\boldsymbol{\mu}}(\boldsymbol{U})\leq &\lim_{n\to\infty} \Gamma_{\e_n}(u_n) \leq \ell E_{\ell^{-1}\alpha}+\frac12V_0\alpha.
\end{aligned}
$$
Then $\boldsymbol{U}\in K_\alpha$. Setting $z_{n,j}=y_n^j+z^j$, we have completed the proof.
\end{proof}

\section{Existence of critical points}

 \subsection{Gradient estimates}
 Let $d_\e>0$ be   such that $d_\e\to 0$ as $\e\to 0$.
By Proposition \ref{pro2.6},  there are $\nu_{L}>0$, $\varepsilon_{L}>0$ such that   if $\varepsilon\in(0,\varepsilon_{L})$, then
 \begin{equation}\label{eq4.1}
  \|\Gamma_{\varepsilon}|_{\mathcal M^\e_\alpha}^{\prime}(u)\|_{*}  \geq 2 \nu_{L},\mbox{ provided that  $u\in  Z(3\rho_0, 2\delta_0  )\setminus Z(\rho_0, 3\delta_0 )\cap [ \Gamma_{\varepsilon}\leq \ell E_{\ell^{-1}\alpha}+\frac12V_0\alpha+2d_\e]$.}
\end{equation}  
 Here we use notation
  $[ \Gamma_{\varepsilon}\leq a]:=\set{u\in \mathcal M^\e_\alpha | \Gamma_{\varepsilon}(u)\leq a}$.
  To prove the existence of a critical point for all small $\e$, we  assume to the contrary that
  \begin{itemize}
   \item[(A)] For any small  $\varepsilon_1\in(0,\varepsilon_{L})$, there exists $\varepsilon\in(0,\varepsilon_1)$ such that  $\Gamma_{\varepsilon}$ has no critical points in
  $  Z (3\rho_0, 3\delta_0)\cap [ \Gamma_{\varepsilon}\leq \ell E_{\ell^{-1}\alpha}+\frac12V_0\alpha+ 2d_{\e}]$.
  \end{itemize}
  Then by Proposition \ref{pro2.5}, there is  $\nu_\e>0$ such that 
   \begin{equation}\label{equ4.2}
    \|\Gamma_{\varepsilon}|_{\mathcal M^\e_\alpha}'(u)\|_{*}\geq 2 \nu_\varepsilon,\mbox{  provided 
    $u\in   Z (3\rho_0, 3\delta_0)\cap [ \Gamma_{\varepsilon}\leq \ell E_{\ell^{-1}\alpha}+\frac12V_0\alpha+ 2d_{\e}]$.}
   \end{equation}

Next we give a gradient estimate when  $\varepsilon \Upsilon_j (u) \in  O^{3\delta_0 }\setminus O^{\delta_0}$ for some $j$. In fact, we show 
\begin{proposition}\label{propo4.1}
  Assume (A).
  Decreasing $\nu_L$ if necessary, it holds that \[\|\Gamma_{\varepsilon}|_{\mathcal M^\e_\alpha}^{\prime}(u)\|_{*} \geq 2 \nu_L \varepsilon\] for all small $\varepsilon$,  provided that 
  $u\in  Z (3\rho_0, 3\delta_0) \cap [ \Gamma_{\varepsilon}\leq \ell E_{\ell^{-1}\alpha}+\frac12 V_0\alpha+2d_{\e}]$ and $\varepsilon \Upsilon_j (u) \in  O^{3\delta_0 }\setminus O^{\delta_0}$ for some $j$.
\end{proposition}
\begin{proof}
To prove Proposition \ref{propo4.1}, we consider $u_\varepsilon\in  Z (3\rho_0, 3\delta_0)$ such that  $\varepsilon \Upsilon_{j_\e} (u_\varepsilon) \in  O^{3\delta_0 }\setminus O^{\delta_0}$ and 
$\|\Gamma_{\varepsilon}|_{\mathcal M^\e_\alpha}^{\prime}(u_\e)\|_{*} =o_\varepsilon(\varepsilon) $ to get a contradiction as $\varepsilon\to 0$.
From  $\|\Gamma_{\varepsilon}|_{\mathcal M^\e_\alpha}^{\prime}(u_\e)\|_{*} =o_\varepsilon(\varepsilon)$ and  Proposition 
\ref{lembound},
it follows   that
\[
    \int_{ \R^N\setminus \cup_{j=1}^\ell B(\Upsilon_j(u_\e),\frac{1}{10}\xi_1( \Upsilon(u_\e)))} \left(|\nabla u_\e|^2+  u_\e^2 \right) 
  \leq  C   e^{-c\xi_1(\Upsilon(u_\e))}+ o_\e(\e). 
 \]
 Hence,
$\Phi_\varepsilon(u_\varepsilon)=0$ and $\Phi_\varepsilon'(u_\varepsilon)=0$ for small $\varepsilon$.

We set 
\[
\begin{aligned}
  \tilde{f_1}(t) = \begin{cases}
-\sigma t\log |t|, & t\in (-e^{-1}, e^{-1}),\\
\sigma e^{-1}\mbox{sgn}(t), & t\in  (-\infty, -e^{-1}]\cup [e^{-1}, +\infty),
  \end{cases}\quad  \tilde{f_2}= \bar f+\tilde f_1,
\end{aligned}  
\]
and $\tilde F_1(u)=\int_0^{|u|} \tilde f_1(t) \rd t$, $\tilde F_2(u)=\int_0^{|u|} \tilde f_2(t) \rd t$, where $\mbox{sgn}$ is the signum.
In $(0,+\infty)$,
$\tilde{f_1}$ is  increasing; $\tilde F_1$ is  convex;  and  by (F5) and $\bar f(t)=\min\{f(K_0),f(t)\}$,
\begin{equation}\label{eqf2}
    |\tilde f_2(t)|\leq Ct,\quad  |\tilde F_2(t)|\leq Ct^2.
\end{equation}
{\bf Step 1}. In this step, we  show the following result.
 \begin{lemma}\label{lemma42}
  There is  a unique $(\lambda_\varepsilon, w_\varepsilon)\in \R\times H_\varepsilon$ such that $\int_{\R^N} w_\e u_\e=\alpha$ and 
  \begin{equation}\label{eq 4.11}
  -\Delta w_\varepsilon+ \widetilde V(\varepsilon x) w_\e +\tilde f_1(w_\e)= \tilde f_2(u_\e) +\lambda_\e u_\e.
  \end{equation}
  Moreover, the following statements hold.
  \begin{enumerate}
    \item
  There is a positive constant $C$ independent of $\varepsilon$ such that $|\lambda_\e|+\|w_\varepsilon\|_\varepsilon + \|\tilde F_1(w_\e)\|_{L^1(\R^N)}\leq C$.
  \item  There hold
\[  
    \|w_\e-u_\e\|_\e+ |\Upsilon_{j_\e} (w_\varepsilon)- \Upsilon_{j_\e} (u_\varepsilon)| =o_\e(\e).
\] 
    \item   $\|\tilde f_2(u_\e)\|_{L^2(\R^N)} $ is bounded, $ w_\e \in H^2_{loc}(\R^N)$ and $ \tilde f_1(w_\e) \in L^\infty(\R^N)$
  \end{enumerate}
\end{lemma}

\begin{proof}
Consider the minimization problem 
  \[ e_\varepsilon= \inf\Big\{ \mathcal I(w)= \frac12 \int_{\R^N} |\nabla w |^2  +\widetilde{V}(\varepsilon x) w^2 +\int_{\R^N} \tilde  F_1(w) - \int_{\R^N} \tilde f_2(u_\e) w \ \Big| \ \int_{\R^N} w u_\e =\alpha \Big\}.
\]
By the compact embedding from $H_\varepsilon$ to $L^q(\R^N)$ for $q\in (\frac{2N}{N+2}, 2^*)$ (\cite[Lemma 2.3]{zhang1}) and the fact that $\mathcal I$ is continuous, convex and coercive,   
   $e_\e$ is attained at some $w_\e$.   The uniqueness follows from the convexity of $\tilde F_1$.
   Since 
\(
e_\e\leq \mathcal I(u_\e)  <C
\)
for some $C>0$, we can conclude from \eqref{eqf2} that 
\[
\frac12 \|w_\e\|_\e^2  + \|\tilde F_1(w_\e)\|_{L^1(\R^N)} \leq C(1+ \|u_\e\|_\e \|w_\e\|_\e) \leq C'(1+\|u_\e\|_\e^2)+\frac14  \|w_\e\|_\e^2.
\]
Then $\|w_\e\|_\e$ is bounded and we can prove (i) by testing \eqref{eq 4.11} by $w_\e$.

To see (ii), we first note that $u_\e-w_\e\in T_{u_\e}\mathcal M^\varepsilon_\alpha$. By assumption (A), $u_\e-w_\e\neq 0$.  
We have by \eqref{eq 4.11} and the monotonicity of $\tilde f_1$ that
\[\begin{aligned}
  o_\e(\e)\|u_\e-w_\e\|_\e=&\Gamma_{\varepsilon}^{\prime}(u_\e)(u_\e-w_\e)\\
  =& (  u_\e, u_\e- w_\e)_\e +\int_{\R^N} \tilde f_1(u_\e) (u_\e -w_\e) -\int_{\R^N} \tilde  f_2(u_\e)(u_\e-w_\e) +  O(e^{-\frac{c}{\e}}) \\
  =&\|u_\e-w_\e\|_\e^2 +\int_{\R^N} (\tilde f_1(u_\e)-\tilde f_1(w_\e))(u_\e-w_\e)+  O(e^{-\frac{c}{\e}}) \\
  \geq & \|u_\e-w_\e\|_\e^2 + O(e^{-\frac{c}{\e}}).
\end{aligned}
\]
Hence, 
$\|u_\e-w_\e\|_\e=o_\e(\e)$. 
By Lemma \ref{lem3.4} (ii),  
$ |\Upsilon_{j_\e} (w_\varepsilon)- \Upsilon_{j_\e} (u_\varepsilon)|=o_\e(\e)$.  
 
 Then,
\begin{align*}
  \Gamma_{\varepsilon}' (v_\e)\varphi=& (v_\e, \varphi)_\e +\int_{\R^N}\tilde f_1(v_\e) \varphi-\int_{\R^N}\tilde f_2(v_\e) \varphi +O(e^{-\frac{c}{\varepsilon}}) \\
  =&(w_\e, \varphi)_\e +\int_{\R^N}\tilde f_1(w_\e) \varphi-\int_{\R^N}\tilde f_2(w_\e) \varphi +O(e^{-\frac{c}{\varepsilon}})+o_\e(\e)\\
  =&\int_{\R^N}(\tilde f_2(u_\e)- \tilde f_2(w_\e))\varphi +\lambda_\e \int_{\R^N} u_\e \varphi +O(e^{-\frac{c}{\varepsilon}}) +o_\e(\e) \\
  =&\lambda_\e \int_{\R^N} (u_\e-w_\e) \varphi +o_\e(\e) =o_\e(\e).
\end{align*}
Therefore, $\|\Gamma_{\varepsilon}|_{\mathcal M_\alpha^\e}' (v_\e)\|_*=o_\e(\e)$.

To see (iii), 
 by \eqref{eqf2}, $\|\tilde f_2(u_\e)\|_{L^2(\R^N)} $ is bounded. Together with $|\tilde f_1|\leq \sigma e^{-1}$, we can 
 get from the 
  elliptic estimate that $w_\e\in H^2_{loc}(\R^N)$.
\end{proof}

   {\bf Step 2.}  
  By Lemma \ref{lemma42} (i) (ii),
 and Proposition 
\ref{lembound},
we have 
\begin{equation}\label{eq000}
    \int_{ \R^N\setminus \cup_{j=1}^\ell B(\Upsilon_j( u_\e),\frac{1}{\sqrt\e})} \left(|\nabla  u_\e|^2+   u_\e^2 + |\nabla  w_\e|^2+   w_\e^2   \right) 
  =o_\e(\e). 
 \end{equation}
 Since $|\tilde f_1(t)t| + |\tilde F_1(t)|\leq C(|t|^\frac{2N}{N+1} +t^2)$, by H\"older inequality,
 we get 
 \begin{equation}\label{eq001}
    \int_{ \frac1\e\Omega\setminus \cup_{j=1}^\ell B(\Upsilon_j( u_\e),\frac{1}{\sqrt\e})} \left( \tilde f_1( u_\e)  u_\e+\tilde f_1( w_\e) w_\e +\tilde F_1( u_\e)+\tilde F_1(  w_\e) \right)=o_\e(1). 
 \end{equation}
Since 
$ \varepsilon \Upsilon_{j_\e} (u_\varepsilon) \in O^{3\delta_0 }\setminus O^{\delta_0} $, up to a subsequence, we may assume that 
$j_\e\equiv 1$,
$  u_\e(\cdot+\Upsilon_j( u_\e))\rightharpoonup u_0\neq 0$,
$\varepsilon \Upsilon_{i} (u_\varepsilon)\to y_i$, $i=1,\cdots,\ell$ and $\frac{\partial{V}}{\partial x_1}(y_1)>\nu_0>0$.
 We take 
  \[\delta_1\in (0, \frac14\min_{y_i\neq y_1} |y_1-y_i|) \quad   \mbox{small enough such that $\displaystyle\frac{\partial{V}}{\partial x_1}>\frac{\nu_0}{2}$ in $\displaystyle B (y_1, 2\delta_1)\subset \Omega$.}\]

Choose  a smooth map  $ \psi_\e \in C_0^\infty(\mathbb R^N,[0,1])$ satisfying 
$|\nabla\psi_\e |\leq 2 \varepsilon/\delta_1 $ and
\begin{equation*}
\psi_\e (x)=
\begin{cases}
1, \  \ &  |x-y_1/\e|\leq   \delta_1  {\varepsilon^{-1}} ,\\
0, \  \ &  |x-y_1/\e| \geq  {2 }\delta_1{\varepsilon^{-1}}.
\end{cases}
\end{equation*}
By \eqref{eq000} and \eqref{eq001},
\[
\int_{\R^N}|\nabla\psi_\e|(|\nabla   u_\e|^2+    u_\e^2+|\nabla   w_\e|^2+   w_\e^2 +\tilde F_1(  u_\e) +\tilde f_1(  u_\e)  u_\e +\tilde F_1(  w_\e) +\tilde f_1( w_\e)   w_\e )=o_\varepsilon(\varepsilon ).  
\]
Testing \eqref{eq 4.11} by 
$\frac{\partial (\psi_\e  w_\e)}{\partial x_1}\in H_\e$, we get 
\[
  \int_{\R^N}\left\{ \nabla   w_\e \nabla \frac{\partial (\psi_\e   w_\e)}{\partial x_1} 
  +\tilde f_1(  w_\e) \frac{\partial (\psi_\e   w_\e)}{\partial x_1} - \tilde f_2(u_\e) \frac{\partial (\psi_\e   w_\e)}{\partial x_1}-\lambda_\varepsilon u_\e \frac{\partial (\psi_\e w_\e)}{\partial x_1}\right\}=-\int_{\R^N}\widetilde{V}(\e x) w_\e \frac{\partial (\psi_\e w_\e)}{\partial x_1}.
\]
Integrating by parts, we get
\[\begin{aligned}
  \int_{\R^N}\nabla w_\e \nabla \frac{\partial (\psi_\e w_\e)}{\partial x_1}  =&\int_{\R^N}\left\{\frac12 \frac{\partial(\psi_\e|\nabla w_\e|^2)}{\partial x_1}+ \frac12 |\nabla w_\e|^2 \frac{\partial \psi_\e}{\partial x_1} + w_\e \nabla w_\e \nabla \frac{\partial \psi_\e}{\partial x_1} +\nabla w_\e \nabla \psi_\e \frac{\partial w_\e}{\partial x_1}\right\}=o_\e(\e), \\
  \int_{\R^N}\tilde f_1(w_\e) \frac{\partial (\psi_\e w_\e)}{\partial x_1} =&\int_{\R^N}\left\{ \frac{\partial(\psi_\e \tilde F_1(w_\e))}{\partial x_1}+
  \frac{\partial \psi_\e}{\partial x_1} [\tilde f_1(w_\e)w_\e -\tilde F_1(w_\e)] 
  \right\}=o_\e(\e),\\
  \int_{\R^N}\tilde f_2(u_\e) \frac{\partial (\psi_\e w_\e)}{\partial x_1}=&\int_{\R^N}\tilde f_2(u_\e) \frac{\partial (\psi_\e u_\e)}{\partial x_1}+o_\e (\e)=\int_{\R^N}\left\{  
  \frac{\partial \psi_\e}{\partial x_1} [\tilde f_2(u_\e)u_\e -\tilde F_2(u_\e)] 
  \right\}+o_\e(\e)=o_\e(\e),\\
  \int_{\R^N} \lambda_\varepsilon u_\e \frac{\partial (\psi_\e w_\e)}{\partial x_1}=&\int_{\R^N} \lambda_\varepsilon w_\e \frac{\partial (\psi_\e w_\e)}{\partial x_1}+o_\varepsilon(\varepsilon)=\frac{\lambda_\varepsilon}2\int_{\R^N} \frac{\partial \psi_\e}{\partial x_1} w_\e^2=o_\varepsilon(\varepsilon),
\end{aligned}
   \]
and
\[
  \int_{\R^N}\widetilde{V}(\e x) w_\e \frac{\partial (\psi_\e w_\e)}{\partial x_1}  =\frac12\int_{\R^N} \left\{
  \frac{\partial (\widetilde V_\e\psi_\e w_\e^2)}{\partial x_1}+  
  \widetilde{V}_\e \frac{\partial\psi_\e}{\partial x_1} w_\e^2 -\frac{\partial \widetilde{V}_\e}{\partial x_1}\psi_\e w_\e^2\right\}=-\frac\e2 \int_{\R^N}\frac{\partial \widetilde{V}(\e x)}{\partial x_1}\psi_\e w_\e^2+o_\e(\e).
\]
Therefore,
\[
  \int_{\R^N}\frac{\partial \widetilde{V}(\e x)}{\partial x_1}\psi_\e w_\e^2 =o_\e(1).  
\]
Taking limits as $\e\to 0$, we have 
\[
\frac{\nu_0}2\int_{\R^N} u_0^2\leq \liminf_{\varepsilon\to 0}  \int_{\R^N}\frac{\partial \widetilde{V}(\e x)}{\partial x_1}\psi_\e w_\e^2 =0.
\]
This is a contradiction.
 \end{proof}
   
  \begin{remark}\label{rek4.3}
    \begin{enumerate}
      \item  To deal with the nonlipschitzian property of the nonlinearity, we have considered the problem in the suitable Hilbert space $H_\e$ to recover the smoothness of energy functional. 
      However, the 
      global $W^{2,p}$ estimate is not applicable for the corresponding operator $ -\Delta +\widetilde V_\e $ for $w_\e$  since $\widetilde{V}_\e$ is unbounded.
      \item We explain how our arguments work for the setting of \cite{Byeon-Tanaka}. In fact, in their setting, there is no restriction on $L^2$ norm of $u_\e$ and   $f(u)/u$ has no singularity,  we can just consider the following equation  to continue our arguments
        \[ 
        -\Delta w_\varepsilon+   V(\varepsilon x) w_\e  =  f(u_\e).
        \]
    \end{enumerate}

  \end{remark} 

 \subsection{Deformation along  negative pseudogradient flow}
By \eqref{eq4.1}, \eqref{equ4.2} and Proposition \ref{propo4.1}, there exists a pseudogradient field on $\mathcal{M}_\alpha^\e$.
\begin{lemma}\label{lemma field}
  There is a locally lipschitzian continuous vector field $\mathcal W:\mathcal{M}_\alpha^\e \to H_\e$ such that the following statements hold.
  \begin{enumerate}
    \item  $\mathcal W(u)\in T_{u}\mathcal M_\alpha^\e$, $\Gamma_\e'(u) \mathcal W(u)\geq 0$ and $\|\mathcal{V}(u)\|_\e \leq 1$ for $u\in \mathcal{M}_\alpha^\e$.
    \item  $ \Gamma_{\varepsilon}' (u) \mathcal W(u)  =0$ if $u\notin  Z (3\rho_0, 3\delta_0 )$.
      \item $ \Gamma_{\varepsilon}' (u) \mathcal W(u)  \geq   \nu_{\e}$,  provided that  $u\in  Z(   2\rho_0, 2\delta_0 ) \cap [   \Gamma_{\varepsilon}\leq \ell E_{\ell^{-1}\alpha}+\frac12V_0\alpha+ d_{\e}]$. 
      \item $ \Gamma_{\varepsilon}' (u) \mathcal W(u)  \geq   \nu_{L}\e$,  provided that  $u\in  Z( 2\rho_0, 2\delta_0 )\setminus Z( \rho_0, \delta_0 ) \cap [   \Gamma_{\varepsilon}\leq \ell E_{\ell^{-1}\alpha}+\frac12V_0\alpha+ d_{\e}]$.
    \item  $ \Gamma_{\varepsilon}' (u) \mathcal W(u)  \geq   \nu_{L}$,  provided that  $u\in  Z(2\rho_0, 2\delta_0  )\setminus Z(\rho_0, 3\delta_0 )\cap [   \Gamma_{\varepsilon}\leq \ell E_{\ell^{-1}\alpha}+\frac12V_0\alpha+ d_\e]$. 
  \end{enumerate}
\end{lemma}
By this lemma, we have 
\begin{lemma}\label{lemma flow}
  Let $\nu_0=\min\{\frac{\delta_0 \nu_L}{4D_2}, \frac{\rho_0 \nu_L}{8} \}$, where $D_2$ is the constant given in Lemma \ref{lem3.4}.
For any $\nu\in(0,\nu_0)$, there is a descending flow $\eta\in C([0,+\infty)\times\mathcal M_\alpha^\e, \mathcal M_\alpha^\e)$ such that 
\begin{enumerate}
  \item $\eta(0, u)=u$,   and $\Gamma_{\varepsilon}(\eta(t,u))\leq \Gamma_\e(u)$   for any $t\in[0,+\infty)$ and $u\in\mathcal M_\alpha^\e$.
  \item For any $t\geq 0$, $\eta(t, u)=u$  provided that $u\notin Z(3\rho_0, 3\delta_0 )$ or $\Gamma_{\varepsilon}(u)\leq \ell E_{\ell^{-1}\alpha}+\frac12V_0\alpha-2\nu$.
  \item For any $t\geq 0$, $\eta(t, u)\in Z (3\rho_0,3\delta_0 )$  if $u\in Z (3\rho_0,3\delta_0 )$.
  \item There is $t_\e>0$ such that $\Gamma_\e(\eta(t_\e,u))<\ell E_{\ell^{-1}\alpha}+\frac12V_0\alpha-\nu$ if $u\in Z (\rho_0,\delta_0 )\cap [\Gamma_{\varepsilon}\leq \ell E_{\ell^{-1}\alpha}+\frac12V_0\alpha+ d_{\e}]$.
\end{enumerate}
\end{lemma}

\begin{proof}
  Let $\psi:\mathcal{M}_\alpha^\e\to [0,1]$ be  locally  Lipschitz continuous such that $\psi(u)=1$ if $\Gamma_\e(u)\geq \ell E_{\ell^{-1}\alpha}+\frac12V_0\alpha-\nu$ 
  and $\psi(u)=0$ if $\Gamma_\e(u)\le \ell E_{\ell^{-1}\alpha}+\frac12V_0\alpha-2\nu$.
For $t\geq 0$, $u\in \mathcal{M}_\alpha^\e$, define $\eta(t,u)$ by the following initial value problem
\[
\frac{\rd}{\rd t}\eta(t, u)= - \psi(\eta(t,u))  \mathcal{W}(\eta(t,u)),\quad  \eta(0,u)=u.
\]
Then (i)  (ii) and (iii) follow  from Lemma \ref{lemma field} (i) and (ii). 
To show (iv), we assume without loss of generality that $d_\e<\nu_0$, and set $t_\e =  \frac{\nu+\nu_0}{\nu_\e} $.
There
are
three cases.

{\bf Case 1.} $\eta(t,u)\in Z (2\rho_0, 2\delta_0 )$ for any $t\in [0,t_\e]$.

In this case, by Lemma \ref{lemma field} (iii),
\begin{align*}
\Gamma_\e(\eta(t_\e,u))\leq& \Gamma_\e(u)+  \int_{0}^{t_\e}  \frac{\rd }{\rd s} \Gamma_\e(\eta(s,u)) \rd s\\
\leq&
\ell E_{\ell^{-1}\alpha}+\frac12V_0\alpha+ d_{\e} - \int_{0}^{t_\e} \Gamma_\e'(\eta(s,u)) \mathcal{W} (\eta(s,u)) \rd s\\
\leq &E_{\ell^{-1}\alpha}+\frac12V_0\alpha+ d_{\e}-\nu_\e t_\e< E_{\ell^{-1}\alpha}+\frac12V_0\alpha -\nu.
\end{align*}

{\bf Case 2.} There is $t\in[0, t_\e]$ such that 
$\dist(\e\Upsilon_j(\eta(t, u)), O)=2\delta_0$ for some $j$ and $\eta(s, u)\in Z( 2\rho_0, 2\delta_0)$ for $s\in[0,t)$.

Let $t_2> t_1>0$ be such that $\dist(\e\Upsilon_j(\eta(t_1, u)), O)=\delta_0$, $\dist(\e\Upsilon_j(\eta(t_2, u)), O)=2\delta_0 $,  and 
$ \eta(t, u)\in  Z( 2\rho_0, 2\delta_0 )\setminus Z( \rho_0, \delta_0 ) $ for $t\in(t_1, t_2)$.
By Lemma \ref{lem3.4}, $|t_1-t_2|\geq \frac{\delta_0}{\varepsilon D_2}$.
Then by Lemma \ref{lemma field} (iv)
\begin{align*}
  \Gamma_\e(\eta(t_\e,u))\leq& \Gamma_\e(u)+  \int_{t_1}^{t_2}  \frac{\rd }{\rd s} \Gamma_\e(\eta(s,u)) \rd s\\
  \leq&
  \ell E_{\ell^{-1}\alpha}+\frac12V_0\alpha+ d_{\e} - \int_{t_1}^{t_2} \Gamma_\e'(\eta(s,u)) \mathcal{W} (\eta(s,u)) \rd s\\
  \leq &E_{\ell^{-1}\alpha}+\frac12V_0\alpha+ d_{\e}-\frac{\delta}{\varepsilon D_2}\nu_L \e< E_{\ell^{-1}\alpha}+\frac12V_0\alpha -\nu.
  \end{align*}
 
  {\bf Case 3.} There is $t\in[0, t_\e]$ such that  $\dist_{H_\e}(\eta(t, u), Z_{L,\e})\geq 2{\rho_0}$, and
$ \e\Upsilon_j(\eta(s, u))\in O^{2\delta_0}$ for any $j$ and $s\in[0,t]$.

In this case, there are $t_2> t_1>0$  such that $\dist_{H_\e}(\eta(t_1, u), Z_{L,\e})= {\rho_0}$, $\dist_{H_\e}(\eta(t_2, u), Z_{L,\e})\geq 2{\rho_0}$,  and 
$ \eta(t, u)\in  Z(2\rho_0, 2\delta_0  )\setminus Z(\rho_0, 2\delta_0 )
= Z(2\rho_0, 2\delta_0  )\setminus Z(\rho_0, 3\delta_0 )$ for $t\in(t_1, t_2)$. Then $\|\eta(t_1, u)-\eta(t_2, u)\|\geq  {\rho_0} $.
By Lemma \ref{lemma field} (i), $|t_1-t_2|\geq \rho_0$.
Then By Lemma \ref{lemma field} (v),
\begin{align*}
  \Gamma_\e(\eta(t_\e,u))\leq& \Gamma_\e(u)+  \int_{t_1}^{t_2}  \frac{\rd }{\rd s} \Gamma_\e(\eta(s,u)) \rd s\\
  \leq&
  \ell E_{\ell^{-1}\alpha}+\frac12V_0\alpha+ d_{\e} - \int_{t_1}^{t_2} \Gamma_\e'(\eta(s,u)) \mathcal{W} (\eta(s,u)) \rd s\\
  \leq &E_{\ell^{-1}\alpha}+\frac12V_0\alpha+ d_{\e}-\rho_0\nu_L  < E_{\ell^{-1}\alpha}+\frac12V_0\alpha -\nu. \qedhere
  \end{align*}

\end{proof}

\subsection{Existence of a critical point}
In this section,  
we  assume (A) and get a contradiction.
Set \[S=\set{\boldsymbol{s}=(s_1,\cdots,s_\ell)\in S_{\ell-1}| |s_j-\ell^{-1}|\leq \delta, \ j=1,\cdots,\ell},\]
where $\delta>0$ is a constant such that $\delta\ell\leq1/2$.
Define
$$\gamma_0(\boldsymbol{p},\boldsymbol{s}):=  B\sum_{j=1}^{\ell}\sqrt{ {\ell s_j} } (\phi_{\e}u_0)(\cdot-p_j)\in  \mathcal M_\alpha^\e,
$$
for each  
$$(\boldsymbol{p}, \boldsymbol{s})\in A(  L):= \Set{\boldsymbol p=(p_1,\cdot\cdot\cdot,p_{\ell_0})\in \big(\frac1\varepsilon  O^{ \delta_0}\big)^{\ell }  |  \xi(\boldsymbol p)\geq   L }\times S,$$
where $B:=\alpha^{1/2} |\sum_{j=1}^{\ell}\sqrt{ {\ell s_j} } (\phi_{\e}u_0)(\cdot-p_j)|_2^{-1}$.
We have the following lemma.
\begin{proposition}\label{pro5.2}
  There is $L_2> L_1$ such that  the following statements hold for
  $L>L_2$   and $\e\in (0,\e_L)$.
\begin{enumerate}
  \item $\gamma_0(\boldsymbol{p},\boldsymbol{s})\in Z(\rho_0, \delta_0 )$ for $(\boldsymbol{p}, \boldsymbol{s})\in A(  L)$.
  \item For any permutation $\sigma$ of $1,2,\cdots,\ell$, 
  $$\gamma_0(p_{\sigma(1)},\cdots,p_{\sigma(\ell)}, s_{\sigma(1)}, \cdots, s_{\sigma(\ell)})=\gamma_0(p_1,\cdots,p_\ell,s_1,\cdots,s_\ell).  
  $$
\item $|p_j-\Upsilon_j(\gamma_0(\boldsymbol p, \boldsymbol s))|\leq 3R_0$ up to a permutation.
\item There is $\nu\in(0,\nu_0)$ independent of $\e$ such that  for any $(\boldsymbol{p}, \boldsymbol{s})\in \partial A(  L)$,
$$\Gamma_{\varepsilon}(\gamma_{0}(\boldsymbol{p}, \boldsymbol{s}))\leq \ell E_{\ell^{-1}\alpha}+\frac12 V_0\alpha-2\nu.
$$
\item 
There is $d_\e>0$ with $d_\e\to 0$ such that 
\[ \sup_{(\boldsymbol{p}, \boldsymbol{s})\in A(L)}\Gamma_{\varepsilon}(\gamma_{0}(\boldsymbol{p}, \boldsymbol{s}))\leq \ell E_{\ell^{-1}\alpha}+\frac12 V_0\alpha+d_\e.
\]
\end{enumerate}
\end{proposition}

\begin{proof}
  (i) follows from the fact that $|B^2-1|\to 0$ uniformly as $L\to\infty$. (ii) and (iii) is clear.
 
To prove (iv), we first note the fact that for large $L>0$, there uniformly holds
$$|\gamma_0(\boldsymbol{p},\boldsymbol{s})|\leq 2\|u_0\|_{L^\infty(\R^N)},\ \ 
(\boldsymbol{p}, \boldsymbol{s})\in A(  L).$$
So by \eqref{e39}, $\bar F(\gamma_0(\boldsymbol{p},\boldsymbol{s}))
=F(\gamma_0(\boldsymbol{p},\boldsymbol{s}))$.
Then we
 consider any sequence $(p(L),s(L))  \in\partial \left(\big(\frac1\varepsilon O^{\delta_0}\big)^\ell\times S\right)$.  
Since $\varepsilon\to 0^+$ as $L\to +\infty$, we have,
 up to a subsequence, $s_j(L)\to s_j$,
$\widetilde V(\varepsilon p_j(L))\to V_j\leq V_0$. 

In the case $(p(L),s(L) )\in\partial \big(\frac1\varepsilon O^{\delta_0}\big)^\ell\times S$,   we have
 $ V_{j_0}\leq \sup V(\partial   O^{\delta_0} )<   V_0$ for some $j_0$.
Therefore,
\begin{equation}\label{eq5.1}
\begin{aligned}
 \limsup_{L\to\infty} \sup_{\varepsilon\in(0, \e_L)}\Gamma_{\varepsilon }(\gamma_0(\boldsymbol{p},\boldsymbol{s}))
&=\mathbb{J}(\sqrt{ {\ell s_1} } u_0,\cdots, \sqrt{ {\ell s_\ell} } u_0)+\sum_{j=1}^{\ell}\frac{V_{j}s_j}2\alpha \\
&\leq \ell E_{\ell^{-1}\alpha}+\frac12 V_0\alpha - \frac{V_{0}-V_{j_0}}{2}s_{j_0}.
\end{aligned}
\end{equation}

When
$(p(L),s(L) )\in  \big(\frac1\varepsilon O^{\delta_0}\big)^\ell\times\partial S$, by \eqref{eq 4444} and similar to \eqref{eq5.1}, 
we have $\limsup_{L\to\infty}  \sup_{\varepsilon\in(0, \e_L)}\Gamma_{\varepsilon }(\gamma_0(\boldsymbol{p},\boldsymbol{s}))<\ell E_{\ell^{-1}\alpha}+\frac12 V_0\alpha$.

Lastly, if $\xi(p)= L$,   setting $u =\gamma_{0}(p,s) $,
we have 
$$\int_{\R^{N} }\chi_u  u^2   \mathrm dx\leq Ce^{-c \xi_1(\Upsilon(u))^2},
$$
for some $C, c>0$ independent of $L, \e$. Then $\Phi_\e(u)=0$ for large $L$. On the other hand, by Corollary \ref{cor2.3}  and
 \eqref{ep}, $\sup_{H_\e}|\Psi_\e|\leq Ce^{-cL^4}$  for some $C, c>0$ independent of $L, \e$. 
Then by the proof of Proposition \ref{prop35},
$\Gamma_{\varepsilon }(\gamma_0(p,s))
\leq  \ell E_{\ell^{-1}\alpha}+\frac12 V_0\alpha-C (L)$ when $L$ is large.

(v) follows from Proposition \ref{prop35} as well.
\end{proof}

As in \cite{Byeon-Tanaka}, we define   an equivalence relation  $\approx$ in $(\R^N)^\ell \times S$ as follows:
\[
(\boldsymbol{p}_1,\cdots \boldsymbol{p}_\ell, s_1,\cdots, s_\ell)  \approx (\boldsymbol{p}'_1,\cdots \boldsymbol{p}'_\ell, s'_1,\cdots, s'_\ell)
\]
if and only if there is a permutation $\sigma$ of $\set{1,\cdots,\ell}$ such that 
$p_{j}=p'_{\sigma(j)}$ and $s_j=s_{\sigma(j)}$ for $j=1,\cdots,\ell$.

Fixing $x_0\in O$, we set
\[
p_j^\varepsilon=\frac{1}{\varepsilon}(x_0+4\sqrt{\varepsilon}(j-1)e_0)\  \ \text{with\ }e_0=(1,0,\cdot\cdot\cdot),
\]
and
$$Q^\varepsilon:=\big[(p_1^\varepsilon,\cdot\cdot\cdot,p_{\ell}^\varepsilon, \ell^{-1},\cdot\cdot\cdot,\ell^{-1})\big]
\in((\R^N)^\ell\times S)/\approx.$$

Define a map $\mathcal F: Z(3\rho_0, 3\delta_0 ) \to ((\R^N)^\ell\times S)/\approx$
$$\mathcal F(u)=\left[\Upsilon_1(u),\cdots,\Upsilon_\ell(u),N_{1,\xi(\Upsilon(u))/4}(u),\cdots,N_{\ell, \xi(\Upsilon(u))/4}(u)\right],
$$
where
$$N_{j,t}(u)= \frac{\int_{B(\Upsilon_j(u), t)}  u^2 }{ \int_{\cup_{j=1}^\ell B(\Upsilon_j(u), t)}u^2} .
$$

By Proposition \ref{pro5.2} (ii) $\mathcal F\circ \gamma_{0}$ can be  considered as a map from  $A(L)/\approx$ to $((\R^N)^\ell\times S)/\approx$.
  
\begin{proposition}\label{pro5.4}
There is $ L_3>L_2$ such that for each $L\geq  L_3$, there hold
$$\deg(\mathcal F\circ \gamma_{0}, A(L)/\approx, Q^\varepsilon)=1.
$$
\end{proposition}
\begin{proof}
  We show that if $L$ is sufficiently large
   \begin{equation}\label{Qe}
    Q^\varepsilon\neq (1-t)[(\boldsymbol p, \boldsymbol s)] + t \mathcal F\circ \gamma_{0}(\boldsymbol p, \boldsymbol s) 
  \end{equation}
   for any $t\in[0,1]$ and 
$(\boldsymbol p, \boldsymbol s)\in \partial A(L)$.
 For $(\boldsymbol p, \boldsymbol s)\in \partial A(L)$, one of the following take   place.
 $$\text{(i)}\ |p_i-p_j|= L\ \text{for some }i\neq j;
 \ \ \ \ \
 \text{(ii)}\ p_j\in\partial \big(\frac{1}{\varepsilon} O^{\delta_0}\big)\ \text{for some }j;
 \ \ \ \ \
 \text{(iii)}\ \boldsymbol s\in \partial S.$$
 
If (i) or (ii) happens, 
by Proposition \ref{pro5.2} (iii), we have 
$\xi((1-t)\boldsymbol p +t\Upsilon(\gamma_{0}(\boldsymbol p, \boldsymbol s)) ) \leq 2L < \frac{4}{\sqrt\e}$ or 
$\dist(\e (1-t) p_j +\e t\Upsilon_j(\gamma_{0}(\boldsymbol p, \boldsymbol s)), x_0) \geq  \delta/2 >  4\ell \sqrt\e$. Hence,
\eqref{Qe} holds.
On the other hand, if (iii) hold,
   by
   $\xi(\Upsilon(u))\geq \xi(\boldsymbol p)-2R_0\geq  L/2$ and the decay estimate for $u_0$, there holds
   \[\lim_{L\to\infty}|N_{j, \xi(\Upsilon(u))/4}( \gamma_0(\boldsymbol p,\boldsymbol s))-s_j| 
=0.\]
Therefore, we can also get \eqref{Qe}. 
\end{proof}

\begin{lemma}\label{lem5.5}
For fixed $L\geq  L_3$, there holds  
 \[
  \liminf_{\varepsilon\to 0}\inf\{\Gamma_{\varepsilon}(u)\ |\ u\in Z (3\rho_0, 3\delta_0),\ \mathcal F(u)=Q^\varepsilon\}\geq\ell E_{\ell^{-1}\alpha}+\frac12 V_0\alpha.
 \]
\end{lemma}
\begin{proof}
For $u $ such that $\mathcal F(u)=Q^\varepsilon$, we have by Lemma \ref{lem1bound},
$\xi(\Upsilon(u)) =\xi(p_1^\varepsilon,\cdot\cdot\cdot,p_{\ell }^\varepsilon)=\frac{4}{\sqrt{\varepsilon}}$.
Note that if $\Gamma_{\varepsilon}(u)\leq \ell E_{\ell^{-1}\alpha}+\frac12 V_0\alpha+1$, we have 
$$ \lim_{\varepsilon\to0}\sum_{j=1}^\ell\int_{\R^{N}\setminus B(p_j^\e,\varepsilon^{-\frac12}) }   u^2 \mathrm dx=0.
$$
By Gagliardo--Nirenberg inequality,
$$ \lim_{\varepsilon\to0} \sum_{j=1}^\ell\int_{\R^{N}\setminus B(p_j^\e,\varepsilon^{-\frac12}) }   \bar F_2(u) \mathrm dx=0  .
$$
Take $\zeta_\e \in C_0^\infty(\R^N,[0,1])$ such that $\zeta_\e =1$ in $B(0,\varepsilon^{-\frac12})$, $\zeta_\e=0$ in $\R^N\setminus B(0,2\e^{-\frac12})$ and $|\nabla\zeta_\e|\leq 10\e^\frac12$.
We have
$$ \lim_{\e\to0}\|\zeta_\e(\cdot-p_j^\e)u-u\|_{L^2(B(p_j^\e,\xi(\Upsilon(u))/4))}^2=0,\quad 
\lim_{\varepsilon\to0}N_{j, \xi(\Upsilon(u))/4}(u)=\ell^{-1}.$$
Moreover,
\[\begin{aligned}
  \int_{B(p_j^\e,\xi(\Upsilon(u))/4)}  |\nabla(\zeta_\e(\cdot-p_j^\e)u)|^2=&\int_{\R^N}|\nabla \zeta_\e(\cdot-p_j^\e)|^2 u^2+\nabla \zeta_\e(\cdot-p_j^\e)\nabla u  \zeta_\e(\cdot-p_j^\e) u +  \zeta_\e(\cdot-p_j^\e) |\nabla u|^2\\
  \leq & \int_{B(p_j^\e,\xi(\Upsilon(u))/4)}  |\nabla u|^2 +o_\varepsilon(1),\\
   \int_{B(p_j^\e,\xi(\Upsilon(u))/4)}F_1(\zeta_\e(\cdot-p_j^\e)u)\leq& \int_{B(p_j^\e,\xi(\Upsilon(u))/4)} F_1(u).
\end{aligned}
\]
Then, we have
\begin{align*}
  \liminf_{\varepsilon\to0}\Gamma_{\varepsilon}(u)
\geq &\liminf_{\varepsilon\to0}\sum_{j=1}^{\ell}\Gamma_{\varepsilon}(\zeta_\e(\cdot-p_j^\e)u)\\
\geq&\liminf_{\varepsilon\to0}\sum_{j=1}^{\ell}J(\zeta_\e(\cdot-p_j^\e)u) + \frac{V_0}{2}\alpha=\ell E_{\ell^{-1}\alpha}+\frac12 V_0\alpha. \qedhere
\end{align*} 
\end{proof}

\begin{proof}[\bf Proof of the existence of critical point of $\Gamma_\e$]
By Proposition \ref{pro5.2} (v), there holds
\[\label{5.9}
 \max_{(p,s)\in A(L)}\Gamma_{\varepsilon}(\gamma_{0}(\boldsymbol p, \boldsymbol s))\leq  \ell E_{\ell^{-1}\alpha}+\frac12 V_0\alpha +d_\e.
\]
By Proposition \ref{pro5.2} (iv), there exists $\nu\in(0,\nu_0)$ such that 
\[\label{5.11}
\max_{(p,s)\in\partial A(L)}\Gamma_{\varepsilon}(\gamma_{0}(\boldsymbol p, \boldsymbol s))\leq \ell E_{\ell^{-1}\alpha}+\frac12 V_0\alpha-2\nu.
\]
If assumption (A) holds. From Lemma \ref{lemma flow},  
\begin{equation}\label{5.12}
\Gamma_{\varepsilon}(\gamma_1(\boldsymbol p, \boldsymbol s)) \leq \ell E_{\ell^{-1}\alpha}+\frac12 V_0\alpha-\nu,\ \ \ \
(p,s)\in A(L),
\end{equation}
where $ \gamma_1(p,s):= \eta(t_\e, \gamma_{0}(\boldsymbol p, \boldsymbol s))$ satisfying $ \gamma_1= \gamma_{0}$ on $\partial A(L)$.  
On the other hand, by Proposition \ref{pro5.4},
$$\deg(\mathcal F\circ  \gamma_1, A(L)/\approx,Q^\varepsilon)=
\deg(\mathcal F\circ\gamma_{0}, A(L)/\approx,Q^\varepsilon)\neq 0,$$
which means that $\mathcal F(\gamma_1(\boldsymbol p_\varepsilon, \boldsymbol s_\varepsilon))=Q^\varepsilon$ for some
$(\boldsymbol p_\varepsilon, \boldsymbol s_\varepsilon)\in A(L)$. By Lemma \ref{lem5.5}, 
$$\liminf_{\varepsilon\to 0}\Gamma_{\varepsilon} ( \gamma_{1}( \boldsymbol p_\varepsilon, \boldsymbol s_\varepsilon))
\geq\ell E_{\ell^{-1}\alpha}+\frac12 V_0\alpha,$$
which contradicts to \eqref{5.12}.  
\end{proof}

\section{Completion of the proof for Theorem \ref{th1.1}}
For each $i$,
we choose  a decreasing sequence  of positive numbers $\{\delta_i\}$, and  a sequence of open sets $\{O_i\}$ such that  $\delta_i\to 0$,   and 
$$O_{i+1}\subset O_{i},\quad  \bigcap_{i=1}^\infty O_i=\mathcal V,\quad  \inf_{O_i^{3\delta_i}\setminus O_i^{\delta_i}}|\nabla V|\geq\tilde\nu_i>0.$$
Then for each $i$, there exist  positive $\nu_i\to0$,
 and positive decreasing $\varepsilon_i\to 0$ such that $\Gamma_\varepsilon$ has a nontrivial critical point 
$(\lambda_{\e,i},u_{\varepsilon,i})\in \mathbb R^N\times {Z (3\rho_0, 3\delta_i )} 
\cap   [ \Gamma_{\varepsilon}\leq \ell E_{\ell^{-1}\alpha}+\frac12 V_0\alpha+2\nu_i]$  when $\varepsilon\in(0,\varepsilon_i)$. Define
$$(\lambda_\e,u_\e)=(\lambda_{\e,i},u_{\e,i}) \ \ \text{for\ }\e\in[\e_{i+1},\e_{i}).
$$
Then for any subsequence of $\varepsilon\to0$,  $u_\varepsilon$ satisfies the assumption of  Proposition \ref{pro2.6}, because
$Z (3\rho_0,\delta_i)\subset Z (  3\rho_0, \delta_0)$ for each $i$.
We have also that $ \e\Upsilon_j(u_\e)\in O_i^{3\delta_i}$ if $\e\in [\e_{i+1},\e_{i})$, $j=1,\cdots,\ell$.
Applying Proposition \ref{pro2.6} to $u_\varepsilon$, Then there exist
 $\boldsymbol{U}\in K_{\alpha}$ and 
 $\left(z_{\e,j} \right) \subset \R^N, j= 1,2, \cdots, \ell$ such that as $\e \to 0$ (after extracting a subsequence if necessary)
 \begin{enumerate}
      \item $ | z_{\e,j}-\Upsilon_j(u_\e)|\leq 2  R_0$ for $j=1,2, \cdots, \ell$,
\item $\|u_{\e}- \sum_{j=1}^\ell U_j (\cdot-z_{\e,j} )\|_{ {\varepsilon}}\to 0$, 
where $U_j$ is the $j$-th component of $\boldsymbol{U}$.
 \end{enumerate}
Then necessarily, for $j=1,\cdots,\ell$,
  \begin{equation}\label{59}
  \dist(\varepsilon z_{\e,j}, O_i^{3\delta_i})\leq \varepsilon|z_{\e,j}-\Upsilon_j(u_\e)|\leq 2R_0\e,\quad \e\in[\e_{i+1},\e_{i}).
  \end{equation}
By the choice of $O_i$ and $\delta_i$, we have 
$\dist(\e z_{\e,j}, \mathcal V)\to 0$ as $\e\to 0$ for $j=1,\cdots,\ell$. Hence, $\boldsymbol{U}$ is a solution to system \eqref{eqsystem} with $\lambda_\e\to\lambda+V	_0$.

By Corollary \ref{lem2.8}, we can conclude that $\Phi_\e(u_\varepsilon)=0$, $\Phi_\e'(u_\varepsilon)=0$.
Hence, $u_\e$ weakly solves
$$-\Delta u+\widetilde V_\e u+\overline V_\e 
T(x,u) u =\bar f(u)+\lambda_\e u,$$
where 
$$T(x,u)=H( e^{\e |x|^2} u)+ \frac12H'( e^{\e |x|^2} |u|)e^{\e |x|^2} |u|,\ \ \ |\lambda_\e|\leq C.$$ 
By Kato's inequality and (F5), for constant $C>0$ independent of $\varepsilon$, $|u_\e|$ weakly solves
$$-\Delta u+\widetilde V_\e u+\overline V_\e 
T(x,u) u \leq \frac12\sigma u\log u+ Cu^{p-1},\ \ \text{for some } p\in(2,2^*).$$
Since $\widetilde V_\e\geq1$, $\overline V_\e\leq 0$,
and $H'(t)\leq 0$ for $t\geq 0$, we have $|u_\varepsilon|$ solves
$$-\Delta u+ u+\overline V_\e H( e^{\e |x|^2} u)u \leq\frac12 \sigma u\log u+Cu^{p-1},\ \ \text{for some } p\in(2,2^*).$$
By this and 
 a   comparison argument (see   \cite[Remark 2.4 (i), Corollary 2.7, Proposition 3.3]{zhang1}),
we have 
$$|u_\varepsilon(x)|\leq C\sum_{j=1}^\ell e^{-c\varepsilon^{-2}|x-z_{\e,j}|^2}, \quad \text{for some}\ C,c>0\ \text{independent of}\ \varepsilon.
$$
Therefore, by \eqref{psi'} and \eqref{59}, $u_\e$ solves
$-\Delta u+ V_\e u=\bar f(u)+\lambda_\e u.$
 By Lemma \ref{lem3.12}, 
 $\|u_\e\|\leq C_0$ and $|\lambda_\e|\leq C_0$.
Since $V(x)\geq 1$ on $B(0, M_0)$, we apply Remark \ref{re3.12} to $|u_\e|$ in $B(0,\e^{-1}M_0)$, and obtain that $|u_\e(x)|\leq K_0$ for $x\in B(0,\e^{-1}M_0-1/2)$.
While   $|u_\e(x)|\leq Ce^{-c\e^{-2}}\leq K_0$, 
$x\not\in B(0,\e^{-1}M_0-1/2)$, for small $\e>0$. Thus $\bar f(u_\e)=f(u_\e)$. 
By Lemma \ref{L2} and the choice of $\rho_1$, $u_\e\ge 0$.
Hence $u_\e$ solves the original problem.
At last, $u_\e>0$ by the maximum principle \cite{Va}. 
 
\section{Appendix}
 
\subsection{Symmetry and decay properties of the autonomous problem}
 
\begin{lemma}\label{decay}
  Assume (F1),   (F4) and (F5).
  Let $u\in H^1(\R^N)$ and $\lambda\in\R$ satisfy
  \[-\Delta u=f(u)+\lambda u,\quad  u\geq 0,\quad u\not\equiv 0.\]
  Then  $u\in C^2(\R^N)$, $u>0$ in $\R^N$, and it is radially symmetric about some point.
 Moreover, if  $|\lambda|+\|u\|_{H^1}\leq C_0$, then 
   there is $C_1>0$, $C_2>0$ independent of $\lambda$ such that 
  \begin{equation}\label{eqAAAA}
    C_1e^{\frac\lambda\sigma}e^{-\frac{\sigma}4r^2}\leq u(r)\leq C_2e^{\frac\lambda\sigma} e^{-\frac{\sigma}4r^2},
  \end{equation}
  and 
  \begin{equation}\label{eqBBB}
  \frac{u'}{r{u}}\to -\frac\sigma2\quad {as}\ r\to+\infty.
  \end{equation}
\end{lemma}
\begin{proof}
  By the maximum principle \cite{Va}, we have $u>0$.
  It is clear that $u\in C^2$ and $|u(x)|+|\nabla u(x)|\to 0$ as $|x|\to \infty$.
To show radial symmetry, we apply moving plane arguments. (See e.g. \cite{GNN, dAvenia2014, Zhang2}.)
Denote $x=(x_1,x')$, and
for $t\in\R$, set 
$$\Sigma_t=\set{x\in\R^N|x_1<t},
$$
$$x_t=(2t-x_1,x'),\quad u_t(x)=u(x_t),\quad w_t=u_t-u.
$$
Then in $\Sigma_{t}$, we have
\begin{equation}\label{wl}
  -\Delta w_t = \lambda w_t
	+f(u_t)-f(u). 
\end{equation}

\textbf{Step 1.} 
By (F5), there is $\tau>0$ such that $f'(s)<-\lambda-1$ for $s\in(0, \tau)$.
Take 
$R>1$ such that $u(x)< \min\{\tau,u(0)\}$ if $|x|\geq R$.
 We show that 
 $w_t \geq 0$ in $\Sigma_t\setminus B_R(0)$ for each $t$.

Otherwise, since $w_t(x)\to 0$ as $  |x|\to+\infty$ and $w_t|_{\partial \Sigma_t}=0$, 
we assume $w_t$ reaches its negative minimum at some $\hat x\in\Sigma_t\setminus B_R(0)$.  
By $ u_t(\hat x)<u(\hat x)$, we have $-\Delta w_t=\lambda w_t
+f(u_t)-f(u)= \int_u^{u_t}(f'(s)+\lambda)\rd s>0$ at $\hat x$. 
This is a contradiction since $\hat x$ is the minimum point of $w_t$. 

Note that Step 1 implies that 
$w_{t}\geq 0$ in $\Sigma_t$ for each $t\leq -R$.

\textbf{Step 2.}
Set $t_0=\sup\set{t   | w_{t'}\geq 0\ \text{in}\ \Sigma_{t'} \ \text{for any}\ t'\in(-\infty,t]}<\infty$.
We claim that $w_{t_0}\equiv 0$.
By continuity, $w_{t_0}\geq 0$.  For $t\leq t_0$
by \eqref{wl}, there holds
\begin{equation}\label{w}
	\begin{aligned}
		-\Delta w_t + \left[\frac{f_1(u_t)-f_1(u)}{u_t-u}-\lambda\right]  w_t =  \frac{f_2(u_t)}{u_t} u_t -\frac{f_2(u)}{u}u\geq \frac{f_2(u)}{u} w_t\geq  0, 
	\end{aligned}
\end{equation}
where $\frac{f_1(u_t)-f_1(u)}{u_t-u}-\lambda$ is bounded from below in $\Sigma_t$. By maximum principle (\cite{Chenli,Trudinger1997}), 
$w_{t}\equiv0$ or 
$w_{t}>0$ in $\Sigma_{t}$. 
If $w_{t_0}\not\equiv0$ then $w_{t_0} >0$.

To finish this step, 
we prove  that there exists $\delta_0>0$ such that for any $\delta\in(0,\delta_0]$
$$w_{\lambda_0+\delta}\geq0 \ \text{in}\ \Sigma_{\lambda_0+\delta}.
$$
Arguing by contradiction, for $\delta_i\to 0^+$, we let $x^i\in \Sigma_{\lambda_0+\delta_i}$ be the negative minimum point of $w_{\lambda_0+\delta_i}$.
  We note that by Step 1, $ |{x^i}|\leq R$ for all $i$. 
We assume $x^i\to x^0$.
Then
$$w_{\lambda_0}(x^0)\leq0,\quad \nabla w_{\lambda_0}(x^0)=0,
$$
 which implies $x^0\in\partial\Sigma_{\lambda_0}$.
 By \eqref{w} and Hopf Lemma (\cite{Chenli,Trudinger1997}), we get a contradiction
 $$\dfrac{\partial w_{\lambda_0}(x^0)}{\partial x_1}<0.
 $$

 Now we have shown that $u_{t_0}= u$ and $\frac{\partial u}{\partial x_1}>0$ in 
 $\Sigma_{t_0}$ by Step 2. 
 Then we can complete the proof since similar arguments hold for any direction in $\R^N$.

 To proceed, we can get
\eqref{eqAAAA} by  comparing with the unique positive solution 
 \[v= e^{\frac{a}\sigma+\frac N2}e^{-\frac{\sigma}4|x|^2}\] 
 to
\begin{equation*}
  \begin{cases}
  -\Delta v=\sigma v\log |v|+av \  \ &\text { in}\ \ \mathbb R^N,\\
  v(x)\to 0 \ \ &\text { as}\ \ |x|\to\infty,
  \end{cases}
  \end{equation*}
  where $a\in \R$. 
  Here we only give the details for the proof of \eqref{eqBBB}.
Set 
$z=-\frac{u'}{ru}$.  
We have 
\[
z'=rz^2 +r^{-1}\frac{f(u)}u-Nr^{-1}z:=d(r,z).
\]
By (F6), as $r\to +\infty$,
 \[
  d(r,z)=rz^2 +r^{-1}\sigma \log u-Nr^{-1}z+O(r^{-1})=r(z^2 -\frac{\sigma^2}4 )-Nr^{-1}z+O(r^{-1}).
  \]
For each $\tau\in(0,1)$, there is $r_{1,\tau}>0$ such that   if $r\geq r_{1,\tau}$ and  $z\geq \frac\sigma{2(1-\tau)}$, then 
\[d(r,z)\geq rz^2(1- (1-\tau)^2 )-Nr^{-1}z+O(r^{-1})\geq z^2. \]
On the other hand, there is $r_{1,\tau}>0$ such that   if $r\geq r_{2,\tau}$ and  $0 <z\leq \frac{\sigma((1-\tau))}{2}$, then 
\[d(r,z)\leq -r \frac{\sigma^2}4 (1- (1-\tau)^2  )-Nr^{-1}z+O(r^{-1})\leq- 1. \]
Once the solution curve $(r,z(r))$ enters   $[r_{1,\tau}, +\infty)\times [\frac\sigma{2(1-\tau)}, +\infty)$ or 
$[r_{2,\tau}, +\infty)\times (0, \frac{\sigma((1-\tau))}{2}]$, it   either blows up  at some finite $r$ or touches the $r-$axis. This is impossible
since $z(r)>0$ exists in $(0,+\infty)$. 
Hence we have  
\[\frac{\sigma((1-\tau))}{2}\leq z(r)\leq \frac\sigma{2(1-\tau)}\quad \mbox{for each}\ r\geq \max\{r_{1,\tau}, r_{2,\tau}\}.
\qedhere\]

\end{proof}

\subsection{Proof of Proposition \ref{prop35}}
 \begin{proof}
  Let $(\lambda, \boldsymbol{v})$ be a solution to \eqref{psa}.   Then by Lemma \ref{lem1bound},  $\sum_{i=1}^\ell(|\lambda|+\|v_i\|_{H^1})$ is bounded. 
Setting $w_j= v_j(\cdot-p_j)$ and $\lambda_i=\lambda-\mu_i$, we have 
\[\begin{aligned}
  J (\sum_{j=1}^\ell   w_j )=&\frac12\int_{\R^N}  | \sum_{j=1}^\ell   \nabla w_j  |^2
-\int_{\R^N}F(\sum_{j=1}^\ell   w_j ) \\
=& \sum_{j=1}^\ell J(w_j)+\frac12\sum_{i=1}^\ell\sum_{j\neq i} \int_{\R^N} \nabla w_i\nabla w_j  +\sum_{i=1}^\ell\int_{\R^N} F(w_i) 
 -\int_{\R^N}F(\sum_{j=1}^\ell   w_j ) \\
= & \mathbb{J}(\boldsymbol{v})+\frac12\sum_{i=1}^\ell \int_{\R^N} (f(w_i)+\lambda_i w_i)\sum_{j\neq i} w_j  +\sum_{i=1}^\ell\int_{\R^N} F(w_i) 
 -\int_{\R^N}F(\sum_{j=1}^\ell   w_j ).  
\end{aligned}\]
Note that by Lemma \ref{F12} (iii),
\[\begin{aligned}
  F(\sum_{j=1}^\ell   w_j )&=\frac{F(\sum_{j=1}^\ell   w_j )}{ (\sum_{j=1}^\ell   w_j )^2}  { (\sum_{j=1}^\ell   w_j )^2}\\
   &=\frac{F(\sum_{j=1}^\ell   w_j )}{ (\sum_{j=1}^\ell   w_j )^2}(\sum_{k=1}^\ell w_k^2 + \sum_{i=1}^\ell\sum_{k\neq i} w_iw_k)\\
  &>\sum_{k=1}^\ell\frac{F(   w_k )}{ w_k^2} w_k^2+\frac{F(\sum_{j=1}^\ell   w_j )}{ (\sum_{j=1}^\ell   w_j )^2}\sum_{i=1}^\ell\sum_{k\neq i} w_iw_k\\
  &=\sum_{k=1}^\ell F(w_k)+\frac{F(\sum_{j=1}^\ell   w_j )}{ (\sum_{j=1}^\ell   w_j )^2}\sum_{i=1}^\ell\sum_{k\neq i} w_iw_k.
\end{aligned}
\]
By (F5), there is $C>0$ such that for $s\in(0, 1+ \max_{1\leq i\leq \ell} \|v_i\|_{L^\infty})$,  
\[
f(s)\leq \sigma s\log s +C s, \quad   F(s)\geq \frac\sigma2 s^2\log s -C s^2.
\]
Therefore,
\[
  J (\sum_{j=1}^\ell   w_j )-  \mathbb{J}(\boldsymbol{v})< \frac12\sum_{i=1}^\ell \sum_{k\neq i}\int_{\R^N} w_iw_k\left(\sigma\log w_i - \sigma\log (\sum_{j=1}^\ell   w_j )+ 4 C \right).
\]

Without loss of generality, for some $i\neq k$, we assume that $|p_i-p_k|=\xi(\boldsymbol p)$, and  up to a transformation of coordinates,   $p_i=(-\xi(\boldsymbol p)/2,0')\in \R^N$ and $p_k=(\xi(\boldsymbol p)/2,0')\in \R^N$ with $0'\in \R^{N-1}$. By Lemma \ref{decay}, 
\[
C_1e^{-\frac{\sigma}4|x-p_i|^2}\leq w_i\leq C_2e^{-\frac{\sigma}4|x-p_i|^2}, \quad C_1e^{-\frac{\sigma}4|x-p_k|^2}\leq w_k\leq C_2 e^{-\frac{\sigma}4|x-p_k|^2}.
\]
Then  we have 
\begin{equation}\label{eq20000}
  \begin{aligned}
    \int_{\R^N} w_i w_k {\rm d}x\leq & C  \int_{\R^N} e^{-\frac\sigma4(|x_1+\frac {\xi(\boldsymbol p)}2|^2+|x_1-\frac {\xi(\boldsymbol p)}2|^2+2|x'|^2)}{\rm d} x_1 {\rm d} x'\\
    =&C \int_{\R^N} e^{-\frac\sigma4(2x_1^2+2|x'|^2+\frac{\xi(\boldsymbol p)^2}2)}{\rm d} x_1 {\rm d} x'
    =C e^{-\frac{\sigma \xi(\boldsymbol p)^2}8},
  \end{aligned}
\end{equation}
where $x=(x_1, x')$ with $x_1\in \R$ and $x'\in \R^{N-1}$.
On the other hand,
 \[\begin{aligned}
   &\int_{\R^N} w_iw_k\left(\log w_i - \log (\sum_{j=1}^\ell   w_j )\right)\\
   \leq& \int_{[0,1]\times \R^{N-1}} w_iw_k\left(\log w_i - \log (\sum_{j=1}^\ell   w_j )\right)   
  \leq    \int_{[0,1]\times \R^{N-1}} w_iw_k \log  \frac{w_i}{w_k}
\\
  \leq &-C\int_{[0,1]\times \R^{N-1}} (|x_1+\frac {\xi(\boldsymbol p)}2|^2-|x_1-\frac {\xi(\boldsymbol p)}2|^2) e^{-\frac\sigma4(|x_1+\frac {\xi(\boldsymbol p)}2|^2+|x_1-\frac {\xi(\boldsymbol p)}2|^2+2|x'|^2)}{\rm d} x_1 {\rm d} x'+C\int_{[0,1]\times \R^{N-1}} w_iw_k\\
  \leq &-C\int_{[0,1]\times \R^{N-1}} 2\xi(\boldsymbol p)x_1 e^{-\frac\sigma 4(2x_1^2+2|x'|^2+\frac{\xi(\boldsymbol p)^2}2)}{\rm d} x_1 {\rm d} x' +C\int_{\R^N} w_iw_k\\
 \leq&-C\xi(\boldsymbol p)e^{-\frac{\sigma \xi(\boldsymbol p)^2}8}\int_{0}^{1} x_1 e^{-\frac\sigma2x_1^2}{\rm d} x_1\int_{\R^{N-1}} e^{- \frac\sigma2|x'|^2} {\rm d} x'+Ce^{-\frac{\sigma \xi(\boldsymbol p)^2}8}.
 \end{aligned}
 \]

By \eqref{eq20000} again, we   deduce
\[
   \alpha < |\sum_{j=1}^\ell w_j|_2^2   \leq  \alpha + C e^{-\frac{\sigma \xi(\boldsymbol p)^2}8},\quad .
\]
Then, 
\[
  J (\sum_{j=1}^\ell   w_j ) +\frac{V_0}2\int_{\R^N} |\sum_{j=1}^\ell   w_j|^2\leq \mathbb{J}(\boldsymbol{v})+\frac{V_0}2\alpha-C'\xi(\boldsymbol p)e^{-\frac{\sigma \xi(\boldsymbol p)^2}8}.
\]
Since $0\leq 1-B^2=|\sum_{j=1}^\ell w_j|_2^{-2}(|\sum_{j=1}^\ell w_j|_2^2-\alpha)\leq C e^{-\frac{\sigma \xi(\boldsymbol p)^2}8}$, we have $|B-1|\leq C e^{-\frac{\sigma \xi(\boldsymbol p)^2}8}$. Hence,
\begin{align*}
  \int_{\R^N} |F(B\sum_{j=1}^\ell w_j)-F(\sum_{j=1}^\ell w_j)|
\leq&|B-1|\int_{\R^N} \sum_{j=1}^\ell w_j |f(\theta\sum_{j=1}^\ell w_j)|\leq C e^{-\frac{\sigma \xi(\boldsymbol p)^2}8},
\end{align*}
where $\theta\in (B,1)$. Then,
\begin{align*}
  J (B\sum_{j=1}^\ell   w_j ) +\frac{V_0}2\int_{\R^N} |B\sum_{j=1}^\ell   w_j|^2 \leq&  J (\sum_{j=1}^\ell   w_j ) +\frac{V_0}2\int_{\R^N} |\sum_{j=1}^\ell   w_j|^2 + C e^{-\frac{\sigma \xi(\boldsymbol p)^2}8}\\
  \leq& \mathbb{J}(\boldsymbol{v})+\frac{V_0}2\alpha-C\xi(\boldsymbol p)e^{-\frac{\sigma \xi(\boldsymbol p)^2}8}. \qedhere
\end{align*}

 \end{proof}

\vskip .1in
 \noindent{\bf Acknowledgement.} 
 The  research was supported by NSFC-12001044, NSFC-12071036, NSFC-11901582.

\end{document}